\documentclass[11pt,a4paper]{amsart}

\usepackage{mymacros_english_paper}
\usepackage{mymacros_common}

\begin{document}
\title[A Higher Algebraic Approach to Liftings of Modules over Derived Quotients]{A Higher Algebraic Approach to Liftings of Modules over Derived Quotients}

\author[R. Ishizuka]{Ryo Ishizuka}
\address{Department of Mathematics, Institute of Science Tokyo, 2-12-1 Ookayama, Meguro, Tokyo 152-8551}
\email{ishizuka.r.ac@m.titech.ac.jp}




\begin{abstract}
    We show a certain existence of a lifting of modules under the self-\(\Ext^2\)-vanishing condition over the ``derived quotient'' by using the notion of higher algebra.
    This refines a work of Auslander--Ding--Solberg's solution of the Auslander--Reiten conjecture for complete interesctions.
    Together with Auslander's zero-divisor theorem, we show that the existence of such \(\Ext\)-vanishing module over derived quotients is equivalent to being local complete intersections.
\end{abstract}

\maketitle 

\setcounter{tocdepth}{1}
\tableofcontents

\section{Introduction}

The \emph{Auslander--Reiten conjecture} was first proposed by Auslander and Reiten in 1975 \cite{auslander1975Generalized}.
This conjecture has posed a significant challenge in the representation theory of algebras and (commutative) ring theory.

\begin{conjecture}[Auslander--Reiten conjecture] \label{AuslanderReiten}
    Let \(R\) be a Noetherian local ring.
    If a finitely generated \(R\)-module \(M\) satisfies \(\Ext_R^i(M, M \oplus R) = 0\) for all \(i \geq 1\), then \(M\) is a projective \(R\)-module.
\end{conjecture}

We suffice to prove the conjecture for \emph{complete} Noetherian local ring \(R\) (see, for example, \cite[Remark 2.3]{kumashiro2023AuslanderReiten}).
To date, the (partial) solutions to this conjecture have been very diverse (see, for instance, \cite[Appendix A.2]{christensen2010Algebras} and the introduction of \cite{kimura2022Maximal,kimura2023Auslander} for recent works).
These solutions mainly focus on the singularity of \(R\).

In particular, Auslander, Ding, and Solberg \cite{auslander1993Liftings} proved the conjecture for any local complete intersection \(R\).
Their strategy is taking a representation \(R \cong A/(f_1, \dots, f_r)\) where \(A\) is a Noetherian local ring and \(f_1, \dots, f_r\) is a regular sequence on \(A\) and providing the following ``lifting property'' of modules:

\begin{theorem}[{\cite[Proposition 1.7]{auslander1993Liftings}\footnote{Also, Yoshino proved this theorem for bounded below complexes whose terms are finite free in \cite[Lemma 3.2]{yoshino1997Theory}.}}] \label{LiftingLemmaADS}
    Let \(A\) be a Noetehrian local ring and let \(f_1, \dots, f_r\) be a regular sequence on \(A\).
    Set \(R \defeq A/(f_1, \dots, f_r)\).
    If a finitely generated \(R\)-module \(M\) satisfies \(\Ext^2_R(M, M) = 0\), then \(M\) is liftable to \(A\), that is, there exists a finitely generated \(A\)-module \(L\) such that \(L \otimes^L_A R \cong M\).
\end{theorem}

If \(A\) is a regular local ring (this is possible by Cohen's structure theorem if a given \(R\) is a complete Noetherian local ring), then the above theorem implies that \(M\) has finite projective dimension, and thus the Auslander--Reiten conjecture holds for the local complete intersection \(R\).

Previously, as this theorem, there are a lot of desirable phenomena if we take quotients by \emph{regular sequences}. However, with the advancement of DG methods and higher algebra, it has become possible to deal with quotients by any sequences of elements (not necessarily regular) by using \emph{derived quotients}; let \(A\) be a Noetherian local ring and let \(f_1, \dots, f_r\) be a (possibly non-regular) sequence of elements of \(A\). The derived quotient \(A/^L\underline{f} \defeq A/^L(f_1, \dots, f_r)\) (\Cref{DefDerivedQuotient}) is a ``commutative algebra object'' in \(D(A)\) and has information not only the usual quotient \(A/(f_1, \dots, f_r)A\) but also their torsions such as \((0:_A f_i)\). In fact, as a complex of \(A\)-modules, \(A/^L(f_1, \dots, f_r)\) is the Koszul complex \(\Kos(f_1, \dots, f_r; A)\) and thus \(A/^L(f_1, \dots, f_r)\) is isomorphic to \(A/(f_1, \dots, f_r)A\) if and only if \(f_1, \dots, f_r\) is a (Koszul-)regular sequence on \(A\).
In this paper, we refine the work of Auslander--Ding--Solberg above (\Cref{LiftingLemmaADS}) by using the notion of higher algebra. This is our first main theorem:

\begin{theorem}[Special case of \Cref{LiftingCorollary}] \label{MainTheorem2Derived}
    Let \(A\) be a Noetherian local ring and let \(f_1, \dots, f_r\) be a sequence of elements of \(A\).
    Set \(A/^L\underline{f} = A/^L (f_1, \dots, f_r)\) and \(R \defeq A/(f_1, \dots, f_r)A\).
    If a finitely generated \(R\)-module \(M\) satisfies \(\Ext_{A/^L\underline{f}}^2(M, M) = 0\), then there exists a finitely generated \(A\)-module \(L\) such that \(M \cong L \otimes^L_A A/^L\underline{f}\).
\end{theorem}

There are two remarks of this theorem.
First, such generalization of \Cref{LiftingLemmaADS} has been studied by using the notion of DG algebras and DG modules (see \Cref{RemarkDG}).
Unlike that, we use the concepts of (stable) \(\infty\)-categories, animated rings, and module spectra introduced by Lurie \cite{lurie2009Higher,lurie2017Higher,lurie2018Spectral} and \v{C}esnavi\v{c}ius--Scholze \cite{cesnavicius2024Purity}.

Second, this theorem is not enough to show the Auslander--Reiten conjecture: in fact it tells us the following ``full circle'' result.

\begin{theorem}[\Cref{EquivLCI}] \label{MainTheorem3}
    Let \(A\) be a regular local ring and let \(f_1, \dots, f_r\) be a sequence of elements of \(A\). As above, set \(A/^L\underline{f} = A/^L (f_1, \dots, f_r)\) and \(R \defeq A/(f_1, \dots, f_r)A\).
    Then the existence of a finitely generated \(R\)-module \(M\) satisfying \(\Ext^2_{A/^L\underline{f}}(M, M) = 0\) is equivalent to that \(f_1,\dots, f_r\) is a regular sequence on \(A\).
\end{theorem}

This theorem is deduced from our theorem (\Cref{MainTheorem2Derived}) and Auslander's zero-divisor theorem (\cite{auslander1961Modules,serre1975Algebre,peskine1973Dimension,hochster1974Equicharacteristic,roberts1987Theoreme}, see \Cref{AuslanderZeroDivisor}).
As a consequence, if we consider the existence of liftings of modules under such a \(\Ext\)-vanishing condition, it is not enough to merely discuss as in \cite{auslander1993Liftings}. See also \Cref{AssumptionLemma2} and \Cref{ExtVanishingExample} where we prove a relationship between the \(\Ext\)-group over derived quotients and the one over usual quotients, and give an example that \(\Ext^i_{A/^L \underline{f}}(M, N)\) does not vanish but \(\Ext^i_{A/(\underline{f})A}(M, N)\) does.


\subsection{Relationship with DG methods}
We end the introduction by recalling the previous work on a generalization of lifting properties of modules by using DG methods and a relationship of their results with our results.

\begin{remark} \label{RemarkDG}
    Let \(\Lambda\) be a (commutative) Noetherian ring and let \(f_1, \dots, f_r\) be a sequence of elements of \(\Lambda\).
    Assume that \(\Lambda\) is \((f_1, \dots, f_r)\Lambda\)-adically complete.
    Let \(A\) be an (associative, unital, graded commutative, and positively graded) DG \(\Lambda\)-algebra such that each term of \(A\) is a finite free \(\Lambda\)-module.\footnote{Any cohomologically finitely generated connective \(\Lambda\)-module is quasi-isomorphic to a complex of finite free \(\Lambda\)-modules. In particular, the underlying complex of any almost perfect animated \(\Lambda\)-algebra is quasi-isomorphic to a complex of finite free \(\Lambda\)-modules by \Cref{EquivAlmostPerfect} and \Cref{NoetherianAlmostPerfect}.}
    In \cite[Corollary 3.5]{nasseh2013Liftings}, they proved that if a DG \(\Kos(\underline{f}; A)\)-module \(D\) that is homologically bounded below and homologically degreewise finite\footnote{In our sense, this means that \(D\) is a \(n\)-truncated almost perfect module for some \(n \geq 0\). See \Cref{EquivAlmostPerfect}.} satisfies \(\Ext^2_{\Kos(\underline{f}; A)}(D, D) = 0\), then there exists a DG \(A\)-module \(E\) such that \(D \cong E \otimes^L_A \Kos(\underline{f}; A)\). 
    Since then, there have been approaches to the lifting problem for polynomial extensions \(A \to A[X_1, \dots, X_n]\) of a DG \(\Lambda\)-algebra and free extensions \(A \to A\abracket{X_1, \dots, X_n}\) of a divided power DG \(\Lambda\)-algebra towards the na\"ive lifting conjecture which leads to the Auslander--Reiten conjecture (see \cite{nasseh2018Weak, ono2021Lifting, nasseh2022Naive} and a more comprehensive discussion in \cite[Appendix A]{nasseh2025Diagonal}).
\end{remark}

Next, we explain how relate our results to the similar results in \Cref{RemarkDG}.

\begin{remark} \label{CompatibleDG}
    By \cite[\S 2.6]{lurie2004Derived}, we have functors of \(\infty\)-categories
    \begin{equation*}
        \CAlg^{an}_{\Lambda} \xrightarrow{\phi} \CAlg^{dg}_{\Lambda} \xrightarrow{\psi} \CAlg^{\mathbb{E}_{\infty}}_{\Lambda}
    \end{equation*}
    where \(\CAlg^{an}_{\Lambda}\) is the \(\infty\)-category of animated \(\Lambda\)-algebras, \(\CAlg^{\mathbb{E}_{\infty}}_{\Lambda}\) is the \(\infty\)-category of \(\mathbb{E}_{\infty}\)-\(\Lambda\)-algebras, and \(\CAlg^{dg}_{\Lambda}\) be the underlying \(\infty\)-category of the ordinary category of commutative differential graded \(\Lambda\)-algebras with a Quillen model structure with quasi-isomorphisms as weak equivalences. Note that any commutative differential graded \(\Lambda\)-algebra is an object of \(\CAlg^{dg}_{\Lambda}\). The composition \(\theta \defeq \psi \circ \phi\) sends an animated ring to its underlying (connective) \(\mathbb{E}_{\infty}\)-ring \(A^{\circ}\) and this functor \(\map{\theta}{\CAlg^{an}_{\Lambda}}{\CAlg^{\mathbb{E}_{\infty}, cn}_{\Lambda}}\) is conservative and commutes all small limits and colimits (\cite[Proposition 2.6.1]{lurie2004Derived} and \cite[Proposition 25.1.2.2]{lurie2018Spectral}).

    If \(\Lambda\) contains \(\setQ\), then \(\psi\) is an equivalence of \(\infty\)-categories and \(\phi\) is fully faithful whose essential image is the connective \(\mathbb{E}_{\infty}\)-\(\Lambda\)-algebras by \cite[Proposition 7.1.4.11 and Proposition 7.1.4.20]{lurie2017Higher}.
    In particular, derived quotients in \(\CAlg^{an}_{\Lambda}\) and Koszul complexes in \(\CAlg^{dg}_{\Lambda}\) define the same objects in \(\CAlg^{\mathbb{E}_{\infty}}_{\Lambda}\) by the universal property (\cite[Lemma 2.3.5]{khan2019Virtual}).
    
    Furthermore, we can compare the \(\infty\)-categories of modules over those objects. First, for any animated ring \(A\), \(\Mod(A)\) and \(\Mod(A^{\circ})\) are the same by the definition (see \Cref{DefMod}). Second, \(\Mod(A^{\circ})\) is equivalent to the underlying \(\infty\)-category of the model category of DG modules over the differential graded \(\Lambda\)-algebra \(\phi(A)\). This is because the model category of modules over the Eilenberg--Mac Lane spectrum \(H\phi(A) \simeq A^{\circ}\) of \(\phi(A)\) is (Quillen-)equivalent to the model category of DG modules over the differential graded \(\Lambda\)-algebra \(\phi(A)\) (see \cite[Theorem 1.2 and Corollary 2.15]{shipley2007HZAlgebra}).
\end{remark}

\begin{remark} \label{DGAniContainsQ}
    Based on \Cref{RemarkDG} and \Cref{CompatibleDG} above, if \(\Lambda\) contains \(\setQ\), there is no difference (in the \(\infty\)-categorical setting) between the \(\infty\)-categories \(\CAlg^{an}_{\Lambda}\) and \(\CAlg^{dg, cn}_{\Lambda}\), and their module categories \(\Mod(A)\) and \(\Mod(\phi(A))\).
    Since our results (\Cref{MainTheorem2Derived} and \Cref{LiftingCorollary}) only treat discrete modules over \(A/^L \underline{f}\), they are a special case of the previous results \cite[Corollary 3.5]{nasseh2013Liftings} explained above if \(\Lambda\) contains \(\setQ\).

    However, we do not know whether our liftability results on the \(\infty\)-category of module spectra \(\Mod(A)\) are compatible with the results on the category of DG modules over the DG \(\Lambda\)-algebra \(\phi(A)\) in general. Otherwise or not, the methods are different: their results in \cite{nasseh2013Liftings} are based on DG algebras and DG modules, in particular, very calculous and concrete although our proof is conceptualized by using higher algebra.
    Also, since our work depends on the discreteness of a given \(M\) (for example in \Cref{DerivedTensor} and \Cref{LiftingExtElement}), it is not clear whether the result \Cref{LiftingCorollary} can be generalized to \(n\)-truncated almost perfect modules as in \cite{nasseh2013Liftings}.
\end{remark}

\subsection*{Acknowledgments}
The author would like to thank Kaito Kimura, Yuya Otake, Kazuma Shimomoto, and Ryo Takahashi for helpful discussions about previous works on the Auslander--Reiten conjecture and related topics, Yutaro Mikami and Masaya Sato for their valuable conversations on the notion of higher algebra, and Saeed Nasseh for his comments on the previous research by using DG methods.
This work was supported by JSPS KAKENHI Grant Number 24KJ1085.

\section{Animated Rings and Their Modules}

In this paper, we freely use the notion of higher algebra such as stable \(\infty\)-categories and animated rings, see for example \cite{lurie2017Higher,cesnavicius2024Purity,bhatt2022Absolute} (or a brief review in \cite[Appendix A]{ishizuka2024Prismatic}).
In this section, we fix some terminology and collect some lemmas about animated rings and their modules which are well-known for experts.

\begin{notation}
    Throughout this paper, we use the following notation.
    Let \(\Lambda\) be a complete Noetherian local ring with the maximal ideal \(\mfrakm\).
    Let \(\CAlg^{an}_\Lambda\) be the \(\infty\)-category of animated (commutative) \(\Lambda\)-algebras.
    Let \(\underline{x} \defeq x_1, \dots, x_t\) be a sequence of elements of \(\mfrakm\).
\end{notation}

\subsection{Basic Notions}

\begin{definition} \label{DefMod}
    Let \(A\) be an animated ring (or more generally, a (connective) \(\mathbb{E}_{\infty}\)-ring).
    The \emph{\(\infty\)-category of \(A\)-modules} is the \(\infty\)-category \(\Mod(A)\) of module spectra over the (underlying) \(\mathbb{E}_{\infty}\)-ring \(A^{\circ}\).\footnote{In this paper, we only mention modules over animated rings instead of ones over \(\mathbb{E}_{\infty}\)-rings. However, this makes no difference: by this definition of \(\Mod(A)\), the module category on an animated ring \(A\) and its underlying \(\mathbb{E}_{\infty}\)-ring \(A^{\circ}\) are the same. So every statements on modules over animated rings are also true for modules over an connective \(\mathbb{E}_{\infty}\)-ring. For a comparison of DG theory, please see \Cref{CompatibleDG}.}
    If \(A\) is discrete, \(\Mod(A)\) is equivalent to the (\(\infty\)-categorical enhancement of the) derived category \(\mcalD(A)\) of \(A\)-modules.
    For any \(M\) and \(N\) in \(\Mod(A)\), we denote by \(\Map_A(M, N)\) the anima of \(A\)-module maps and by \(\mapspt_A(M, N)\) the \(A\)-module spectra of \(A\)-module maps.
    This \(\infty\)-category \(\Mod(A)\) is stable and presentable (in particular, it has small colimits small limits by \cite[Corollary 7.1.1.5, Corollary 4.2.3.7, and Corollary 4.2.3.3]{lurie2017Higher}).
\end{definition}

\begin{definition}
    Let \(A\) be an animated ring and let \(M\) be an \(A\)-module.
    The \emph{\(i\)-th homotopy group} \(\pi_i(M)\) of a module \(M\) over an animated ring \(A\) is that of the underlying spectra of \(M\). Since we use the homological notation in this paper, \(\pi_i(-)\) corresponds to the \((-i)\)-th cohomology group \(H^{-i}(-)\): in fact, \(\pi_i(-)\) defines a (canonical) \(t\)-structure \((\Mod(A)_{\geq 0}, \Mod(A)_{\leq 0})\) on \(\Mod(A)\) (see \cite[Proposition 7.1.1.13]{lurie2017Higher}).
    An \(A\)-module \(M\) is \emph{connective} (or \emph{animated}) if \(\pi_i(M) = 0\) for any \(i < 0\).
    A \emph{discrete} \(A\)-module \(M\) is an object of \(\Mod(A)\) such that \(\pi_i(M) = 0\) for any \(i \neq 0\).
    Note that \emph{\(A\)-modules} in this paper is always an object of \(\Mod(A)\) not of the abelian category of discrete \(A\)-modules.
\end{definition}

We recall the \(\Ext\)-group in the \(\infty\)-category of \(A\)-modules.

\begin{definition}[{\cite[Notation 7.1.1.11 and Remark 7.1.1.12]{lurie2017Higher}}]
    Let \(A\) be an animated ring.
    For \(A\)-modules \(M\) and \(N\), we denote by \(\Ext^i_A(M, N)\) the discrete \(\pi_0(A)\)-module \(\pi_0(\Map_A(M, N[i]))\) or \(\pi_{-i}(\mapspt_A(M, N))\) for \(i \in \setZ\).
    If \(A\) is a discrete ring and \(M\) and \(N\) are complexes of \(A\)-modules, \(\mapspt_A(M, N)\) is isomorphic to the derived functor \(R\Hom_A(M, N)\) of the internal hom and then the \(A\)-module \(\Ext^i_A(M, N)\) is isomorphic to the usual Yoneda \(\Ext\)-group, that is, \(\Ext^i_A(M, N) \cong H^i(R\Hom_A(M, N))\).
\end{definition}

\begin{definition}[{\cite[Definition 7.2.4.30]{lurie2017Higher} and \cite[Definition 2.4]{kerz2018Algebraic}}] \label{DefNoetherianAnimated}
    Let \(A\) be an animated ring.
    We say that \(A\) is \emph{coherent} if \(\pi_0(A)\) is a coherent ring and each homotopy group \(\pi_i(A)\) is a finitely presented \(\pi_0(A)\)-module.
    We say that \(A\) is \emph{Noetherian} if \(\pi_0(A)\) is Noehterian and \(A\) is coherent.
\end{definition}

\begin{definition}[{\cite[Definition 7.2.4.1 and Proposition 7.2.4.2]{lurie2017Higher}}]
    Let \(A\) be an animated ring.
    The \(\infty\)-category \(\Mod(A)^{\perf}\) is the smallest stable subcategory of \(\Mod(A)\) that contains \(A\) and is closed under retracts.
    We say that an \(A\)-module \(M\) is \emph{perfect} if it belongs to \(\Mod(A)^{\perf}\).
    This is equivalent to that \(M\) is compact in \(\Mod(A)\).
\end{definition}

\begin{definition}[{\cite[Definition 7.2.4.10]{lurie2017Higher} and \cite[Definition 2.7.0.1]{lurie2018Spectral}}] \label{DefAlmostPerfect}
    Let \(A\) be an animated ring and let \(M\) be an \(A\)-module.
    \begin{enumalphp}
        \item \(M\) is \emph{perfect to order \(n\)}  if, for every filtered diagram \(\{N_\alpha\}\) in \(\Mod(A)_{\leq 0}\), the map of \(\pi_0(A)\)-modules
        \begin{equation*}
            \colim_\alpha \Ext_A^i(M, N_\alpha) \to \Ext_A^i(M, \colim_\alpha N_\alpha)
        \end{equation*}
        is bijective for all \(i < n\) and injective for \(i = n\).
        \item \(M\) is \emph{almost perfect} if it is perfect to order \(n\) for every integer \(n\). This notion is equivalent to pseudo-coherence of \(M\) in the sense of \citeSta{064Q} if \(A\) is discrete.
    \end{enumalphp}
    Especially, if \(M\) is connective, then \(M\) is perfect to order \(0\) if and only if \(\pi_0(M)\) is a finitely generated \(\pi_0(A)\)-module (\cite[Proposition 2.7.2.1(1)]{lurie2018Spectral}).
    By tensor-forgetful adjunction, the property of being perfect to order \(n\) is stable under the (derived) base change.
\end{definition}

\begin{lemma}[{\cite[Corollary 2.7.2.3]{lurie2018Spectral}}] \label{EquivAlmostPerfect}
    Let \(A\) be a Noetherian animated ring and let \(M\) be an \(A\)-module.
    Then \(M\) is perfect to order \(n\) if and only if the following conditions hold:
    \begin{enumalphp}
        \item \(M\) is bounded below, that is, for any sufficiently small \(m \ll 0\), \(\pi_m(M) = 0\).
        \item For any \(m \leq n\), \(\pi_m(M)\) is a finitely generated \(\pi_0(A)\)-module.
    \end{enumalphp}
\end{lemma}





\begin{definition}[{\cite[Definition 7.2.4.21]{lurie2017Higher}}]
    Let \(A\) be an animated ring and let \(I \subseteq \pi_0(A)\) be an ideal.
    For \(a \leq b \in \setZ \cup \{\pm \infty\}\), an \(A\)-module \(M\) has \emph{\(I\)-complete Tor-amplitude in \([a, b]\)} if \(M \otimes^L_A N\) is contained in \(\Mod(A)_{[a, b]}\) for any \(I^{\infty}\)-torsion (or equivalently, \(I\)-torsion) discrete \(A\)-module \(N\).
    We say that \(M\) has \emph{Tor-amplitude \(\leq n\)} if \(M\) is \(0\)-complete Tor-amplitude in \([- \infty, n]\), that is, for every discrete \(A\)-module \(N\), \(N \otimes^L_A M\) is \(n\)-truncated
\end{definition}

\subsection{Behavior under base change}

The properties in the previous subsection are stable under base change under some conditions (\Cref{LemBaseChangePerfectNil} and \Cref{LemBaseChangePerfect}). To show this, we need the following lemma.

\begin{lemma} \label{StableCompletenessTensor}
    Let \(A\) be a discrete ring with a weakly proregular ideal \(I \subseteq A\) and let \(M\) be a derived \(I\)-complete \(A\)-module which is contained in \(\Mod(A)_{[a, b]}\).
    For any almost perfect \(A\)-module \(N\), the \(A\)-module \(M \otimes^L_A N\) is derived \(I\)-complete.
\end{lemma}

\begin{proof}
    By \cite[Theorem 0.3 and Theorem 4.7]{yekutieli2025Derived} (or \cite{porta2014Homology,porta2015Cohomologically}), there exists a complex \(M_{\bullet}\) of derived \(I\)-complete (discrete) \(A\)-modules which represents \(M\).
    Since \(N\) is almost perfect, there exists a (homologically) bounded below complex \(N_{\bullet}\) of finite free \(A\)-modules which represents \(N\) (\citeSta{064Q}).
    Then the derived tensor product \(M \otimes^L_A N\) is represented by the totalization of the double complex \(M_{\bullet} \otimes_A N_{\bullet}\).
    If we set \(N_n = A^{\oplus k_n}\) for some integer \(k_n \geq 0\), the \(n\)-th term of the totalization is
    \begin{equation} \label{TotalizationCompletion}
        \Tot_n(M_{\bullet} \otimes_A N_{\bullet}) = \bigoplus_{n = i + j} (M_i \otimes_A N_j) = \bigoplus_{n = i + j} M_i^{\oplus k_j}.
    \end{equation}
    By \(M \in \Mod(A)_{[a, b]}\), \(M_{\bullet}\) is quasi-isomorphic to \(M_{\bullet}' \defeq \tau_{\leq b}\tau_{\geq a}M_{\bullet}\).
    Since the category of derived \(I\)-complete discrete \(A\)-modules consists of an abelian category, the complex \(M_{\bullet}'\) is a bounded complex of derived \(I\)-complete \(A\)-modules.
    This shows that the \(n\)-th term (\ref{TotalizationCompletion}) of \(M \otimes^L_A N\) is a finite direct sum of derived \(I\)-complete \(A\)-modules and thus \(M \otimes^L_A N\) is derived \(I\)-complete.
\end{proof}

As in the flatness, we can show the following.

\begin{lemma}[{cf. \cite[Lemma 5.15]{bhatt2021CohenMacaulayness}}] \label{NoetherianAdicallyTorAmpCompletion}
    Let \(A\) be an animated Noetherian ring with a finitely generated ideal \(I \subseteq \pi_0(A)\) and let \(M\) be a derived \(I\)-complete \(A\)-module.
    Suppose that \(\pi_0(A)\) is derived \(I\)-complete and \(M\) has \(I\)-complete Tor-amplitude in \([a, b]\).
    Then \(M\) has Tor-amplitude in \([a, b]\).
\end{lemma}

\begin{proof}
    We have to show that \(M \otimes^L_A N \cong (M \otimes^L_A \pi_0(A)) \otimes^L_{\pi_0(A)} N\) is contained in \(\Mod(A)_{[a, b]}\) for any discrete \(A\)-module \(N\).
    So we can assume that \(A\) is discrete.

    Approximating \(N\) by a discrete finitely generated \(\pi_0(A)\)-module, we may assume that \(N\) is a discrete almost perfect \(A\)-module by \Cref{EquivAlmostPerfect}.
    Since \(M\) is derived \(I\)-complete and \(I\)-complete Tor-amplitude in \([a, b]\), \(M = R\lim_n(M \otimes^L_A A/I^n)\) is a bounded \(A\)-module. Using \Cref{StableCompletenessTensor}, we have \(M \otimes^L_A N\) is derived \(I\)-complete.
    Since pro-systems \(\{N \otimes^L_A A/I^n\}_{n \geq 1}\) and \(\{N/I^nN\}_{n \geq 1}\) are pro-isomorphic by the Artin--Rees lemma for Noetherian rings, we have
    \begin{equation*}
        M \otimes^L_A N \cong (M \otimes^L_A N)^{\wedge} = R\lim_n (M \otimes^L_A N \otimes^L_A A/I^n) \cong R\lim_n (M \otimes^L_A N/I^nN).
    \end{equation*}
    By the Milnor exact sequence (see, for example, \cite[Lemma F.233]{gortz2023Algebraic}), we have the following short exact sequence of \(A\)-modules for each \(k \in \setZ\):
    \begin{equation*}
        0 \to R^1\lim_n \pi_{k+1}(M \otimes^L_A N/I^nN) \to \pi_k(M \otimes^L_A N) \to \lim_n \pi_k(M \otimes^L_A N/I^nN) \to 0.
    \end{equation*}
    Since \(M\) is \(I\)-complete Tor-amplitude in \([a, b]\), \(M \otimes^L_A N/I^nN\) is in \(\Mod(A)_{[a, b]}\) for all \(n \geq 1\).
    By this assumption, \(M \otimes^L_A N\) is in \(\Mod(A)_{[a-1, b]}\).
    The exact sequence \(0 \to I^nN/I^{n+1}N \to N/I^{n+1}N \to N/I^nN \to 0\) induces an exact sequence \(\pi_a(M \otimes^L_A N/I^{n+1}N) \to \pi_a(M \otimes^L_A N/I^nN) \to \pi_{a-1}(M \otimes^L_A I^nN/I^{n+1}N) = 0\).
    The Mittag--Leffler condition shows that \(R^1\lim_n \pi_a(M \otimes^L_A N/I^nN)\) vanishes and thus \(M \otimes^L_A N\) is in \(\Mod(A)_{[a, b]}\).
\end{proof}

The following lemma is one of the stability of some properties under base change.

\begin{lemma}[{\cite[Proposition 2.7.3.2]{lurie2018Spectral}}] \label{LemBaseChangePerfectNil}
    Let \(A \to B\) be a surjective\footnote{A map \(\map{f}{A}{B}\) of animated rings is \emph{surjective} if its connected component \(\map{\pi_0(f)}{\pi_0(A)}{\pi_0(B)}\) is a surjective map of usual rings.} map of animated rings with the nilpotent kernel \(I\) of \(\pi_0(A) \to \pi_0(B)\).
    Let \(M\) be a connective \(A\)-module and set a connective \(B\)-module \(M_B \defeq M \otimes^L_A B\).
    Then we have the following.
    \begin{enumalphp}
        \item \(M\) is perfect to order \(n\) over \(A\) if and only if \(M_B\) is perfect to order \(n\) over \(B\).
        \item \(M\) is almost perfect over \(A\) if and only if \(M_B\) is almost perfect over \(B\).
        \item \(M\) has Tor-amplitude \(\leq k\) over \(A\) if and only if \(M_B\) has Tor-amplitude \(\leq k\) over \(B\).
        \item \(M\) is perfect over \(A\) if and only if \(M_B\) is perfect over \(B\).
        \item For each \(n \in \setZ\), \(M\) is \(n\)-connective if and only if \(M_B\) is \(n\)-connective. 
    \end{enumalphp}
\end{lemma}

In \Cref{LemBaseChangePerfectNil} above, we need to assume that \(I\) is nilpotent.
However, by using the topological Nakayama's lemma (see, for example, \cite[Theorem 8.4]{matsumura1986Commutative}) instead of the usual Nakayama's lemma, some proof works under our assumption without the nilpotency of \(I\) and we can show the following ``topological'' variant.

\begin{lemma}[{Topological variant of \cite[Proposition 2.7.3.2]{lurie2018Spectral}}] \label{LemBaseChangePerfect}
    Let \(A \to B\) be a surjective map of animated rings with the kernel \(I\) of \(\pi_0(A) \to \pi_0(B)\).
    Let \(M\) be a connective \(A\)-module and set a connective \(B\)-module \(M_B \defeq M \otimes^L_A B\).
    Assume that \(\pi_0(A)\) is \(I\)-adically complete Noetherian and the \(\pi_0(A)\)-module \(\pi_0(M)\) is \(I\)-adically separated.
    Then we have the following.
    \begin{enumalphp}
        \item \(M\) is perfect to order \(n\) over \(A\) if and only if \(M_B\) is perfect to order \(n\) over \(B\).
        \item \(M\) is almost perfect over \(A\) if and only if \(M_B\) is almost perfect over \(B\).
        \item If \(M\) has Tor-amplitude \(\leq k\) over \(A\), then \(M_B\) has Tor-amplitude \(\leq k\) over \(B\).
        \item If \(M\) is perfect over \(A\), then \(M_B\) is perfect over \(B\). If \(A\) is Noetherian, the converse also holds.
        \item If \(\pi_i(M)\) is \(I\)-adically separated for all \(i \geq 0\), \(M\) is \(n\)-connective if and only if \(M_B\) is \(n\)-connective for each \(n \in \setZ\).
    \end{enumalphp}
\end{lemma}

\begin{proof}
    The only if part of (a) is the tensor-forgetful adjunction and the definition of `perfect to order \(n\)' (\Cref{DefAlmostPerfect}).
    As in the proof of \cite[Proposition 2.7.3.2]{lurie2018Spectral}, we show the if part by induction on \(n\).
    The case of \(n < 0\) is clear.
    If \(n \geq 0\), we have an isomorphism \(\pi_0(M_B) \cong \pi_0(M)/I\pi_0(M)\) and \(\pi_0(M_B)\) is finite over \(\pi_0(B) \cong \pi_0(A)/I\) by \cite[Proposition 2.7.2.1 (1)]{lurie2018Spectral}.
    Since \(\pi_0(M)\) is \(I\)-adically separated and \(\pi_0(A)\) is \(I\)-adically complete, then \(\pi_0(M)\) is also a finitely generated \(\pi_0(A)\)-module by topological Nakayama's lemma.
    If \(n = 0\), this shows the claim; \(M\) is perfect to order \(0\) over \(A\) (\cite[Proposition 2.7.2.1 (1)]{lurie2018Spectral}).

    For general \(n > 0\), one can take a fiber sequence \(N \to A^k \to M\) of \(A\)-modules for some \(k \geq 1\) with a connective \(A\)-module \(N\) such that \(\pi_0(A)^k \to \pi_0(M)\) is surjective. Here we use the above argument to show that \(\pi_0(M)\) is finite over \(\pi_0(A)\).
    Taking \(- \otimes^L_A B\), we have a fiber sequence \(N \otimes^L_A B \to B^k \to M_B\) of \(B\)-modules.
    Since \(M_B\) is perfect to order \(n\), then \(N \otimes^L_A B\) is perfect to order \(n-1\) by \cite[Proposition 2.7.2.1 (2)]{lurie2018Spectral}. 
    Since \(\pi_0(A)\) is Noetherian, the finite \(\pi_0(A)\)-module \(\pi_0(N)\) is \(I\)-adically separated and we can apply the inductive hypothesis for \(N\), namely, \(N\) is perfect to order \(n-1\) over \(A\).
    By using \cite[Proposition 2.7.2.1 (2)]{lurie2018Spectral} again, \(M\) is perfect to order \(n\).

    The equivalence (b) is by (a) and the definition of almost perfectness.

    We next prove (c). By the isomorphism \(M_B \otimes^L_B N \cong M \otimes^L_A N\) for any (discrete) \(B\)-module \(N\).

    The conditions (b) and (c) give the first assertion of (d) by \cite[Proposition 7.2.4.23 (4)]{lurie2017Higher}.
    We show the converse if \(A\) is Noetherian. Again by using \cite[Proposition 7.2.4.23 (4)]{lurie2017Higher}, \(M_B\) is almost perfect and has Tor-amplitude \(\leq k\) over \(B\) for some \(k\).
    By (b), \(M\) is almost perfect over \(A\) and it suffices to show that \(M\) has Tor-amplitude \(\leq k\) over \(A\).
    Any \(I\)-torsion discrete \(A\)-module \(N\) can be regarded as a \(\pi_0(B) \cong \pi_0(A)/I\)-module and also as a \(B\)-module.
    By the assumption, \(M \otimes^L_A N \cong M_B \otimes^L_B N\) is \(k\)-truncated and thus \(M\) has \(I\)-complete Tor-amplitude \(\leq k\) over \(A\).
    Since \(M\) is almost perfect and \(\pi_0(A)\) is \(I\)-adically complete Noetherian, \(\pi_n(M)\) is derived \(I\)-complete for any \(n \in \setZ\) and so is \(M\) by \cite[Theorem 7.3.4.1]{lurie2018Spectral}.
    By \Cref{NoetherianAdicallyTorAmpCompletion}, \(M\) has Tor-amplitude \(\leq k\) over \(A\) and this shows the if part of (d).

    Finally, we prove (e).
    The only if part is the stability of (\(n\)-)connectivity under tensor products (\cite[Corollary 7.2.1.23]{lurie2017Higher}).
    The proof of the if part goes by induction on \(n\).
    The case of \(n \leq 0\) is clear since \(M\) is connective.
    If \(n > 0\) and \(M_B\) is \(n\)-connective, then \(M\) is \((n-1)\)-connective by the inductive hypothesis.
    This shows that \(0 \cong \pi_{n-1}(M_B) \cong \pi_{n-1}(M)/I\pi_{n-1}(M)\) and thus \(\pi_{n-1}(M) = 0\) by topological Nakayama's lemma since \(\pi_{n-1}(M)\) is \(I\)-adically separated.
\end{proof}

\begin{lemma} \label{NoetherianAlmostPerfect}
    Let \(A\) be an animated \(\Lambda\)-algebra.
    Then \(A\) is almost perfect as an \(\Lambda\)-module if and only if \(A\) is Noetherian and \(\pi_0(A)\) is finite over \(\Lambda\).
    In particular, if \(A\) is almost perfect over \(\Lambda\), then \(\pi_n(A)\) is complete with respect to the maximal ideal \(\mfrakm\) of \(\Lambda\) for all \(n \in \setZ\) and in particular \(A\) is derived \(\mfrakm\)-complete.
\end{lemma}

\begin{proof}
    If \(A\) is almost perfect as an \(\Lambda\)-module, then \(\pi_m(A)\) is finite over \(\Lambda = \pi_0(\Lambda)\) for all \(m \in \setZ\) by \Cref{EquivAlmostPerfect}.
    In particular, \(\pi_0(A)\) is Noetherian and \(\pi_m(A)\) is finite over the \(\Lambda\)-algebra \(\pi_0(A)\). This shows \(A\) is Noetherian.
    Conversely, if \(A\) is Noetherian and \(\pi_0(A)\) is finite over \(\Lambda\), then \(\pi_m(A)\) is finite over \(\Lambda\) for all \(m \in \setZ\). By using \Cref{EquivAlmostPerfect} again, \(A\) is almost perfect as an \(\Lambda\)-module.

    The last assertion follows from the fact that \(\pi_m(A)\) is finite over \(\Lambda\) for all \(m \in \setZ\) since any homologically \(\mfrakm\)-adically complete \(\Lambda\)-module is derived \(\mfrakm\)-complete (for example \cite[Theorem 7.3.4.1]{lurie2018Spectral}).
\end{proof}

In the last of the proof of our main theorem, we use the following lemmas (\Cref{EquivProjective} and \Cref{AlmostPerfMapFin}) which are basic results in the discrete case.

\begin{lemma} \label{EquivProjective}
    Let \(A\) be a coherent animated ring and let \(M\) be a connective perfect \(A\)-module.
    Then the following are equivalent.
    \begin{enumerate}
        \item \(M\) is projective over \(A\) in the sense of \cite[Proposition 7.2.2.6]{lurie2017Higher}.
        \item The \(\Ext\)-module \(\Ext^i_A(M, A)\) vanishes for \(i \geq 1\).
    \end{enumerate}
\end{lemma}

\begin{proof}
    The projectivity of \(M\) over \(A\) is equivalent to the vanishing of \(\Ext^i_A(M, Q) = 0\) for any discrete (or, equivalently, connective) \(A\)-module \(Q\) and for any \(i \geq 1\) by \cite[Proposition 7.2.2.6]{lurie2017Higher}. So (1) implies (2).
    We will show (2) \(\Rightarrow\) (1).
    We can assume that \(Q\) is a discrete finitely generated \(A\)-module, in particular, \(Q\) is connective and almost perfect over \(A\) by \cite[Proposition 7.1.1.13 (3) and Proposition 7.2.4.17]{lurie2017Higher}.
    As in the proof of \cite[Proposition 7.2.4.11 (5)]{lurie2017Higher}, \(Q\) can be represented as the colimit of a sequence of \(A\)-modules
    \begin{equation*}
        0 = Q(-1) \xrightarrow{f_0} Q(0) \xrightarrow{f_1} Q(1) \xrightarrow{f_2} Q(2) \to \cdots
    \end{equation*}
    such that \(Q(j)\) is connective, \(\cofib(f_j)[-j]\) is a free \(A\)-module \(A^{\oplus n_j}\) of finite rank \(n_j\) for \(j \geq 0\).\footnote{More generally, by the same proof of \cite[Proposition 7.2.4.11]{lurie2017Higher}, we can show that any connective \(A\)-module \(Q\) is the colimit of a sequence \(0 = Q(-1) \xrightarrow{f_0} Q(1) \xrightarrow{f_1} Q(2) \xrightarrow{f_2} \cdots\) of connective \(A\)-modules such that \(\cofib(f_j)[-j]\) is a free \(A\)-module for any \(j \geq 0\). This representation has the same spirit as the semifree resolution of DG \(A\)-modules (see for example \cite[\S 2.12]{nasseh2022Naive}).}
    Applying \(\mapspt_A(M, (-))\) for the (co)fiber sequence \(Q(j-1) \xrightarrow{f_j} Q(j) \to \cofib(f_j)\) for each \(j \geq 0\), we have an exact sequence
    \begin{equation*}
        \Ext^i_A(M, Q(j-1)) \to \Ext^i_A(M, Q(j)) \to \Ext^i_A(M, \cofib(f_j)) \cong (\Ext^{i+j}_A(M, A))^{\oplus n_j}
    \end{equation*}
    for each \(i \geq 1\) by virtue of the equivalence \(\cofib(f_j) \cong A^{\oplus n_j}[j]\).
    Since \(\Ext^k_A(M, A)\) vanishes for any \(k \geq 1\) by the assumption, \(\Ext^i_A(M, Q(j-1)) \xrightarrow{\cong} \Ext^i_A(M, Q(j))\) is an isomorphism for each \(i, j \geq 1\).
    Since \(M\) is perfect over \(A\), the following isomorphisms hold for each \(i \geq 1\):
    \begin{align*}
        \Ext^i_A(M, Q) & \cong \Ext^i_A(M, \colim_j Q(j)) \cong \colim_j \Ext^i_A(M, Q(j)) \cong \Ext^i_A(M, Q(0)).
    \end{align*}
    By the construction of the sequence \(\{Q(j)\}_{j \geq -1}\), \(Q(0) \cong \cofib(f_0)[-0]\) is a finite free \(A\)-module.
    Our assumption, the vanishing of \(\Ext^i_A(M, A)\) for \(i \geq 1\), shows that \(\Ext^i_A(M, Q)\) vanishes for any \(i \geq 1\) and thus \(M\) is projective over \(A\).
\end{proof}

\begin{lemma} \label{AlmostPerfMapFin}
    Let \(A\) be an animated Noetherian ring and let \(L\) and \(L'\) be almost perfect \(A\)-modules.
    If \(L'\) is \(k\)-truncated, then the \(\Ext\)-module \(\Ext^i_A(L, L')\) is finitely generated over \(\pi_0(A)\) for any \(i \in \setZ\).
    In particular, if \(L\) is an almost perfect discrete \(A\)-module, then \(\Ext^i_A(L, L)\) is finitely generated over \(\pi_0(A)\) for any \(i \in \setZ\).
\end{lemma}

\begin{proof}
    Fix an integer \(n\). Our assumption implies that \(L\) is perfect to order \(n\).
    By \cite[Corollary 2.7.2.2]{lurie2018Spectral}, there exists a fiber sequence \(F \to P \to L\) such that \(P\) is a perfect \(A\)-module and \(F\) is \(n\)-connective (that is, \(\pi_i(F) = 0\) for \(i < n\)).
    Taking \(\mapspt_A(-, L')\) and its long exact sequence, we have an exact sequence of \(\pi_0(A)\)-modules
    \begin{align*}
        \Ext^{i-1}_A(F, L') \to \Ext^i_A(L, L') \to \Ext^i_A(P, L') \to \Ext^i_A(F, L').
    \end{align*}
    Considering the \(t\)-structure on \(\Mod(A)\) (or simply on \(\mathrm{h}\Mod(A)\)) and the \(n\)-connectivity of \(F\), the first and the last terms vanish for \(i \leq n-k-1\) since \(L'\) is \(k\)-truncated.
    We have to show that \(\Ext^i_A(P, L')\) is finitely generated over \(\pi_0(A)\).
    Since \(P\) is perfect over \(A\), the \(A\)-module \(\mapspt_A(P, L')\) is equivalent to \(P^{\vee} \otimes^L_A L'\) for some perfect \(A\)-module \(P^{\vee}\) by \cite[Proposition 7.2.4.4]{lurie2017Higher}.
    Since \(L'\) and \(P^{\vee}\) are almost perfect over \(A\), so is the \(A\)-module \(P^{\vee} \otimes^L_A L'\).
    In particular, the homology group \(\pi_{-i}(P^{\vee} \otimes^L_A L') \cong \pi_{-i}(\mapspt_A(P, L')) = \Ext^i_A(P, L')\) is finite over \(\pi_0(A)\) for any \(i \in \setZ\) by \Cref{EquivAlmostPerfect} and so is \(\Ext^i_A(L, L')\) for \(i \leq n-k-1\) by the above exact sequence.
    This \(n\) can be taken arbitrarily large and we have the desired result.
    The second statement follows from the first statement by taking \(L' = L\) and \(k = 0\).
\end{proof}

\section{Derived Quotients}

In this section, we recall the notion of derived quotients and calculate some homotopy groups of (derived) tensor products (\Cref{TorsionCalculation} and \Cref{DerivedTensor}).
Recall that \((\Lambda, \mfrakm)\) is a complete Noetherian local ring.

\begin{definition} \label{DefDerivedQuotient}
    Let \(A\) be an animated ring and let \(\underline{f} = f_1, \dots, f_r\) be a sequence of elements of \(\pi_0(A)\).
    Take an \(A\)-module \(M\). 
    We denote by \(M/^L\underline{f} = M/^L(f_1, \dots, f_r)\) the \emph{derived quotient} of \(M\) by \(\underline{f}\), that is,
    \begin{equation*}
        M/^L\underline{f} \defeq M \otimes^L_{\setZ[X_1, \dots, X_t]} \setZ \cong M \otimes^L_A (A/^L\underline{f}) \in \Mod(A/^L\underline{f}),
    \end{equation*}
    where \(M \leftarrow \setZ[X_1, \dots, X_t] \to \setZ\) is the map \(\times x_i \mapsfrom X_i \mapsto 0\).
    Note that \(A/^L \underline{f}\) is an animated \(A\)-algebra.
    If \(A\) is discrete, the underlying complex of \(M/^L \underline{x}\) is the Koszul complex \(\Kos(M; \underline{x}) \defeq M \otimes^L_A \Kos(A; \underline{x})\) in \(D(A)\).

    By considering the exact sequence \(0 \to \setZ[X] \xrightarrow{\times X} \setZ[X] \to \setZ \to 0\), we have a fiber sequence in \(\Mod(A)\)
    \begin{equation}
        M \xrightarrow{\times x} M \to M/^L x. \label{FiberSequenceModx}
    \end{equation}
    Moreover, if \(B\) is an animated \(A\)-algebra, then the derived quotient \(B/^L(f_1, \dots, f_r)\) has a natural structure of animated \(B\)-algebras and a morphism \(B \to B/^L(f_1, \dots, f_r)\) of animated rings since it is defined by a homotopy cocartesian square in \(\CAlg^{an}_{\setZ}\) (see, for example, \cite[2.3.1]{khan2019Virtual}).
\end{definition}

\begin{notation}
    Let \(A\) be an animated \(\Lambda\)-algebra and let \(x\) be an element of \(\mfrakm\).
    For \(n \geq 1\), we set the derived quotient
    \begin{equation*}
        A_n \defeq A/^L x^n = A \otimes^L_{\setZ[X]} \setZ \in \CAlg^{an}_\Lambda.
    \end{equation*}
    Then we have maps of animated \(\Lambda\)-algebras
    \begin{equation*}
        A \to A_1 \leftarrow A_2 \leftarrow \dots \leftarrow A_n \leftarrow \cdots.
    \end{equation*}
    If \(A\) is Noetherian, so is \(A_n\) by \Cref{DefNoetherianAnimated}. Note that this canonical map \(A_{n+1} \to A_n\) depends on \(n\) and \(x\). Even if \(x = 0\) and \(A\) is discrete, this map is not an isomorphism in general (consider the map on \(\pi_1\)).
\end{notation}


\begin{lemma} \label{LemDerivedQuotient}
    Let \(A\) be an animated \(\Lambda\)-algebra and let \(x\) be an element of \(\mfrakm\).
    Let \(M\) be an \(A\)-module.
    Then we have an isomorphism of \(A\)-modules:
    \begin{equation*}
        M/^L x^n \cong M \otimes^L_{\setZ[X]} \setZ[X]/(X^n),
    \end{equation*}
    where \(M \leftarrow \setZ[X] \to \setZ[X]/(X^n)\) is the map \(\times x \mapsfrom X \mapsto \overline{X}\).
    In particular, the following isomorphisms hold:
    \begin{align*}
        A_m/^L x^n & \cong A_m \otimes^L_A A_n \in \CAlg^{an}_\Lambda \\
        M/^L x^n & \cong M \otimes^L_A A_n \in \Mod(A_n)
    \end{align*}
    for any \(m, n \geq 0\). In particular, \(A_i\) is isomorphic to \(A \otimes^L_{\setZ[X]} \setZ[X]/(X^i)\) as an \(A_n\)-module via the canonical map \(A_n \cong A \otimes^L_{\setZ[X]} \setZ[X]/(X^n) \to A \otimes^L_{\setZ[X]} \setZ[X]/(X^i) \cong A_i\) for \(1 \leq i \leq n-1\).
\end{lemma}

\begin{proof}
    Taking \(- \otimes^L_{\setZ[X]} M\) for the fiber sequence \(\setZ[X] \xrightarrow{\times X^n} \setZ[X] \to \setZ[X]/(X^n)\), we have a fiber sequence of \(A\)-modules
    \begin{equation*}
        M \xrightarrow{\times x^n} M \to M \otimes^L_{\setZ[X]} \setZ[X]/(X^n).
    \end{equation*}
    This shows \(M/^L x^n \cong M \otimes^L_{\setZ[X]} \setZ[X]/(X^n)\).
    By this isomorphism, \(A_n \cong A \otimes^L_{\setZ[X]} \setZ[X]/(X^n)\) holds. So we have the following isomorphism of \(A_n\)-modules:
    \begin{align*}
        M \otimes^L_A A_n & \cong M \otimes^L_A (A \otimes^L_{\setZ[X]} \setZ[X]/(X^n)) \cong M \otimes^L_{\setZ[X]} \setZ[X]/(X^n) \cong M/^L x^n.
    \end{align*}
    If \(M\) is an animated \(\Lambda\)-algebra \(A_m\), the isomorphism \(A_m/^L x^n \cong A_m \otimes^L_A A_n\) in \(\Mod(A_n)\) becomes an isomorphism between animated \(\Lambda\)-algebras.
\end{proof}

\begin{lemma} \label{DistinguishedTriangleAn}
    Let \(A\) be an animated \(\Lambda\)-algebra and let \(x\) be an element of \(\mfrakm\).
    Then there exists a fiber sequence in \(\Mod(A_{n+1})\)
    \begin{equation*}
        A_k \xrightarrow{\times x^{n+1-k}} A_{n+1} \to A_{n+1-k}
    \end{equation*}
    for any \(n \geq 1\) and \(1 \leq k \leq n\).
\end{lemma}

\begin{proof}
    By the isomorphism \(A_n \cong A \otimes^L_{\setZ[X]} \setZ[X]/(X^n)\) in \Cref{LemDerivedQuotient}, we have a fiber sequence in \(\Mod(A_{n+1})\)
    \begin{equation*}
        A \otimes^L_{\setZ[X]} \setZ[X]/(X^k) \xrightarrow{\times X^{n+1-k}} A \otimes^L_{\setZ[X]} \setZ[X]/(X^{n+1}) \to A \otimes^L_{\setZ[X]} \setZ[X]/(X^{n+1-k}).
    \end{equation*}
    This gives the desired fiber sequence.
\end{proof}

\begin{lemma} \label{TorsionCalculation}
    Let \(A\) be an animated \(\Lambda\)-algebra and let \(x\) be an element of \(\mfrakm\).
    For any discrete \(A\)-module \(M\), we have
\begin{equation*}
    \Tor_k^A(A_n, M) \defeq \pi_k(A_n \otimes^L_A M) \cong 
    \begin{cases}
        M/x^nM & \text{if } k = 0 \\
        M[x^n] & \text{if } k = 1 \\
        0 & \text{otherwise},
    \end{cases}
\end{equation*}
where \(M[x^j] \defeq \set{m \in M}{x^jm = 0}\) is the \(x^j\)-torsion submodule of \(M\) for \(j \geq 0\).
\end{lemma}

\begin{proof}
    We have a fiber sequence \(A \xrightarrow{\times x^n} A \to A_n\) of \(A\)-modules.
    Then \(A_n \otimes^L_A M\) is the cofiber of \(M \xrightarrow{\times x^n} M\) in \(\Mod(A)\).
    Therefore, we have the desired calculation.
\end{proof}

The following calculation is a generalization of the same calculation for a \emph{regular element} \(x\) proved in \cite[Lemma 1.1]{auslander1993Liftings} and a key to the proof of \Cref{LiftingEquivFiberSeq}.
The proof is given first, and the necessary preparations are described afterwards (\Cref{ColimitPreserving}--\Cref{GeometricRealizationisComplex}).

\begin{lemma} \label{DerivedTensor}
    Let \(A\) be an animated \(\Lambda\)-algebra and let \(x\) be an element of \(\mfrakm\).
    Let \(M\) be a discrete \(A_n\)-module.
    For each \(1 \leq i \leq n-1\), we have the following isomorphisms of discrete \(\pi_0(A_n)\)-modules:
    \begin{equation*}
        \pi_k(M \otimes^L_{A_n} A_i) \cong 
        \begin{cases}
            M/x^iM & \text{if } k = 0 \\
            M[x^i]/x^{n-i}M & \text{if \(k\) is odd} \\
            M[x^{n-i}]/x^iM & \text{if \(k \neq 0\) is even}.
        \end{cases}
    \end{equation*}
\end{lemma}

\begin{proof}
    The \(\setZ[X]\)-module \(\setZ[X]/(X^i)\) is quasi-isomorphic to the complex of \(\setZ[X]\)-modules
    \begin{equation*}
        Z \defeq \cdots \xrightarrow{\times X^i} \setZ[X]/(X^n) \xrightarrow{\times X^{n-i}} \setZ[X]/(X^n) \xrightarrow{\times X^i} \setZ[X]/(X^n) \to 0.
    \end{equation*}
    Set the complex \(Z = \{Z_k\}_{k \in \setZ}\) and the quasi-isomorphism \(Z \to \setZ[X]/(X^i)\) in \(\Ch(\setZ[X])\).
    By the Dold--Kan correspondence \(\map{\DK}{\Ch_{\geq 0}(\Ch(\setZ[X]))}{\Fun(\Delta^{\opposite}, \Ch(\setZ[X]))}\), we have a simplicial object \(\DK_{\bullet}(Z')\) in \(\Ch(\setZ[X])\) for the object
    \begin{equation} \label{ResolZ}
        Z' \defeq (\cdots \xrightarrow{\times X^{n-i}} \setZ[X]/(X^n)[0] \xrightarrow{\times X^i} \setZ[X]/(X^n)[0] \to 0)
    \end{equation}
    in \(\Ch_{\geq 0}(\Ch(\setZ[X]))\).
    Applying \Cref{GeometricRealizationisComplex} for the simplicial \(\setZ[X]\)-module \(\DK_{\bullet}(Z')\), the geometric realization \(\abs{\DK_{\bullet}(Z')}\) in \(\mcalD(\setZ[X])\) is isomorphic to the totalization of \(Z' \in \Ch_{\geq 0}(\Ch(\setZ[X]))\), namely,
    \begin{equation} \label{DoldKanTotalization1}
        \abs{\DK_{\bullet}(Z')} \cong \iota_{\setZ[X]}(Z) \cong \setZ[X]/(X^i)
    \end{equation}
    in \(\mcalD(\setZ[X]) = \Mod(\setZ[X])\) where \(\iota_{\setZ[X]}\) is defined in \Cref{ColimitPreserving}.

    By \Cref{LemDerivedQuotient}, \(A_i\) is isomorphic to \(A \otimes^L_{\setZ[X]} \setZ[X]/(X^i)\) as an \(A_n\)-module.
    The isomorphism (\ref{DoldKanTotalization1}) is \(\Mod(\setZ[X])\) provides the following isomorphisms of \(A_n\)-modules:
    \begin{align*}
        M \otimes^L_{A_n} A_i & \cong M \otimes^L_{A_n} (A \otimes^L_{\setZ[X]} \setZ[X]/(X^i)) \cong M \otimes^L_{A_n} (A \otimes^L_{\setZ[X]} Z) \\
        & \cong M \otimes^L_{A_n} (A \otimes^L_{\setZ[X]} \abs{\DK_{\bullet}(Z')}) \cong \abs{M \otimes^L_{A_n} (A \otimes^L_{\setZ[X]} \DK_{\bullet}(Z'))}.
    \end{align*}
    Recall that the simplicial \(\setZ[X]\)-module \(\DK_{\bullet}(Z')\)\footnote{In \cite[Construction 1.2.3.5]{lurie2017Higher}, the value of the functor \(\DK\) is an \emph{additive} category. However, since any stable \(\infty\)-category \(\mcalC\) has finite colimits, we can construct the simplicial object \(\DK_{\bullet}(Z')\) in \(\mcalC\) for each non-negatively chain complex with a value \(\mcalC\).} is defined by
    \begin{equation*}
        \DK_n(Z') \defeq \bigoplus_{\alpha \colon [n] \twoheadrightarrow [k]} Z'_k \in \Mod(\setZ[X]),
    \end{equation*}
    where \(Z'_k = \setZ[X]/(X^n)[0]\) is the \(k\)-th term of the complex \(Z' \in \Ch_{\geq 0}(\Ch(\setZ[X]))\) (\ref{ResolZ}) and the sum is taken over all surjections \(\alpha \colon [n] \twoheadrightarrow [k]\) in the simplex category \(\Delta\).
    So the \(n\)-th term of the simplicial \(A_n\)-module \(M \otimes^L_{A_n} (A \otimes^L_{\setZ[X]} \DK_{\bullet}(Z'))\) can be written as
    \begin{align}
        & M \otimes^L_{A_n} (A \otimes^L_{\setZ[X]} \DK_n(Z')) = M \otimes^L_{A_n} (A \otimes^L_{\setZ[X]} \bigoplus_{\alpha \colon [n] \twoheadrightarrow [k]} Z_k) \label{DoldKanResolution} \\
        & \cong \bigoplus_{\alpha \colon [n] \twoheadrightarrow [k]} (M \otimes^L_{A_n} (A \otimes^L_{\setZ[X]} Z_k)) \cong 
        \bigoplus_{\alpha \colon [n] \twoheadrightarrow [k]} M. \nonumber
    \end{align}
    Here, we use the isomorphism \(A_n \cong A \otimes^L_{\setZ[X]} \setZ[X]/(X^n) = A \otimes^L_{\setZ[X]} Z_k\) of \(A_n\)-modules described in \Cref{LemDerivedQuotient} for each \(k \geq 0\).
    The map \(\map{\beta^*}{\DK_n(Z')}{\DK_{n'}(Z')}\) for \(\map{\beta}{[n']}{[n]}\) in \(\Delta\) is defined by an appropriate sum of the identity map on \(Z_k\) and the differential, \(\times X^{n-i}\) and \(\times X^i\).
    Through the above canonical isomorphism (\ref{DoldKanResolution}), the simplicial \(A_n\)-module \(M \otimes^L_{A_n} (A \otimes^L_{\setZ[X]} \DK_{\bullet}(Z'))\) is isomorphic to \(\DK_{\bullet}(\cdots \xrightarrow{\times x^{n-i}} M \xrightarrow{\times x^i} M \to 0)\), that is, we have an isomorphism in \(\Fun(\Delta^{\opposite}, \Mod(A_n))\);
    \begin{equation}
        M \otimes^L_{A_n} (A \otimes^L_{\setZ[X]} \DK_{\bullet}(Z')) \cong \DK_{\bullet}(\cdots \xrightarrow{\times x^{n-i}} M \xrightarrow{\times x^i} M \to 0). \label{DoldKanIsom}
    \end{equation}
    The right-hand side is the simplicial \(A_n\)-module corresponds to the non-negatively chain complex of \(A_n\)-modules \((\cdots \xrightarrow{\times x^{n-i}} M \xrightarrow{\times x^i} M \to 0) \in \Ch_{\geq 0}(\Mod(A_n))\).
    This complex is the image of the double complex of discrete \(\pi_0(A_n)\)-modules \((\cdots \xrightarrow{\times x^{n-i}} M[0] \xrightarrow{\times x^i} M[0] \to 0)\) via the (chain complex of the) canonical functor
    \begin{equation} \label{RestScalar}
        \Ch(\pi_0(A_n)) \xrightarrow{\iota_{\pi_0(A_n)}} \Mod(\pi_0(A_n)) \to \Mod(A_n),
    \end{equation}
    which maps \(M[0]\) to \(M\) and preserves coproducts.
    The former functor \(\iota_{\pi_0(A_n)}\) is defined by \Cref{ColimitPreserving} and the latter one is the restriction of scalars.
    Set the simplicial object \(F \defeq \DK_{\bullet}(\cdots \xrightarrow{\times x^{n-i}} M[0] \xrightarrow{\times x^i} M[0] \to 0)\) in \(\Ch(\pi_0(A_n))\).
    The right-hand side of (\ref{DoldKanIsom}) is the restriction of scalar to \(A_n\) of the simplicial \(\pi_0(A_n)\)-module \(\iota_{\pi_0(A_n)}(F)\).
    We need to calculate the colimit \(\abs{\iota_{\pi_0(A_n)}(F)} = \colim_{\Delta^{\opposite}} \iota_{\pi_0(A_n)}(F)\) in \(\mcalD(\pi_0(A_n))\).
    Applying \Cref{GeometricRealizationisComplex} for \(F\), the geometric realization \(\abs{\iota_{\pi_0(A_n)}(F)}\) in \(\mcalD(\pi_0(A_n))\) is isomorphic to
    \begin{equation*}
        \abs{\iota_{\pi_0(A_n)} \DK_{\bullet}(\cdots \xrightarrow{\times x^{n-i}} M[0] \xrightarrow{\times x^i} M[0] \to 0)} \cong \iota_{\pi_0(A_n)}(\cdots \xrightarrow{\times x^{n-i}} M \xrightarrow{\times x^i} M \to 0).
    \end{equation*}
    By using the colimit preserving property of the restriction of scalars (\cite[Corollary 4.2.3.7 (2)]{lurie2017Higher}), we have the following isomorphism in \(\mcalD(A_n)\):
    \begin{align*}
        M \otimes^L_{A_n} A_i & \cong \abs{M \otimes^L_{A_n} (A \otimes^L_{\setZ[X]} \DK_{\bullet}(Z))} \cong \abs{\DK_{\bullet}(\cdots \xrightarrow{\times x^{n-i}} M \xrightarrow{\times x^i} M \to 0)} \\
        & \cong \iota_{\pi_0(A_n)}(\cdots \xrightarrow{\times x^{n-i}} M \xrightarrow{\times x^i} M \to 0)
    \end{align*}
    by using the above calculation and (\ref{DoldKanIsom}).
    This shows that the \(i\)-th homotopy group of \(M \otimes^L_{A_n} A_i\) is the desired one.

\end{proof}

In the above proof, we use the following observation.

\begin{remark} \label{ColimitPreserving}
    Let \(A\) be a discrete ring.
    Recall that the \(\infty\)-category of \(A\)-modules \(\Mod(A)\) is equivalent to the derived \(\infty\)-category \(\mcalD(A) \defeq \dgNerve(\Ch(A)^{\circ})\) of (discrete) \(A\)-modules (see \cite[Definition 1.3.5.8 and Remark 7.1.1.16]{lurie2017Higher}).
    When emphasizing the model-theoretical construction of \(\mcalD(A)\), we use the notation \(\mcalD(A)\) instead of \(\Mod(A)\).

    By \cite[Proposition 1.3.5.15]{lurie2017Higher}, the functor \(\Ch(A) \to \dgNerve(\Ch(A)) \xrightarrow{L} \mcalD(A)\) exhibits \(\mcalD(A)\) as the underlying \(\infty\)-category of the model category \(\Ch(A)\) with a model structure (\cite[Proposition 1.3.5.3]{lurie2017Higher}).
    Namely, the functor of \(\infty\)-categories
    \begin{equation*}
        \iota_A \colon \Ch(A) \to \dgNerve(\Ch(A)) \xrightarrow{L} \mcalD(A)
    \end{equation*}
    induces the equivalence of \(\infty\)-categories \(\Ch(A)[W^{-1}] \cong \mcalD(A)\) where \(W\) is the class of quasi-isomorphisms in \(\Ch(A)\).

\end{remark}

To calculate the geometric realization, we use the following construction. Our notation is based on \cite[Problem 4.24]{bunke2013Differential} and \cite[Definition 2.2]{arakawa2025Homotopy}.

\begin{construction} \label{TotSimplicial}
    Let \(A\) be a discrete ring and let \(F\) be a simplicial object in \(\Ch(A)\).
    The \emph{totalization} \(\tot(F)\) of \(F\) is the totalization \(\Tot^{\oplus}(M_*(F))\) of the Moore complex \(M_*(F)\) of the simplicial chain complex \(F\).
    Explicitly, the totalization \(\tot(F)\) is defined as follows:
    \begin{equation*}
        \tot(F)_n \defeq \Tot^{\oplus}(M_*(F))_n = \bigoplus_{n = p+q} M_*(F)_{p, q} = \bigoplus_{n = q+p} F([q])_p
    \end{equation*}
    with the differential
    \begin{equation*}
        d_{p, q} \defeq (-1)^q d^{F([q])}_p + \sum_{i=0}^q (-1)^i\partial^{F([q])_p}_i \colon F([q])_p \to F([q])_{p-1} \oplus F([q-1])_p
    \end{equation*}
    where \(\map{d^{F([q])}_p}{F([q])_p}{F([q])_{p-1}}\) is the differential of the chain complex \(F([q]) \in \Ch(A)\) and \(\map{\sum_{i=0}^q (-1)^i \partial^{F_p([q])}_i}{F([q])_p}{F([q-1])_p}\) is the alternating sum of the face maps \(\{\partial^{F_p([q])}_i\}\) of the simplicial discrete \(A\)-module \(F_p \in \Fun(\Delta^{\opposite}, \Ch(A)^{\heartsuit})\).
\end{construction}

By using this, we have the following lemma which is a special case of \cite[Proposition 2.4 (1) and Lemma 2.8]{arakawa2025Homotopy} but we recall the proof in our case.

\begin{lemma} \label{GeometricRealizationisTotalization}
    Let \(A\) be a discrete ring and let \(F\) be a simplicial object in \(\Ch(A)\).
    Under the notation in \Cref{ColimitPreserving}, we can show the following isomorphism holds:
    \begin{equation*}
        \colim_{\Delta^{\opposite}} \iota_A F \cong \iota_A(\tot(F))
    \end{equation*}
    in \(\mcalD(A)\).
\end{lemma}

\begin{proof}
    Recall that the restriction of scalar \(\mcalD(A) \to \mcalD(\setZ)\) is faithful (on the homotopy categories) and preserves small colimits by \cite[Corollary 4.2.3.7 and Remark 7.1.1.16]{lurie2017Higher}.
    It suffices to show that the image of the complex \(\tot(F)\) in \(\mcalD(\setZ)\) is equivalent to the colimit of the diagram \(\iota_\setZ F \colon \Delta^{\opposite} \xrightarrow{\iota_A F} \mcalD(A) \to \mcalD(\setZ)\).
    This diagram is the same as the composition \(\iota_\setZ F \colon \Delta^{\opposite} \xrightarrow{F} \Ch(A) \to \Ch(\setZ) \xrightarrow{\iota_{\setZ}} \mcalD(\setZ)\) and it suffices to calculate this colimit.
    By \cite[Problem 4.24]{bunke2013Differential} (and its proof), we have the isomorphism
    \begin{equation*}
        \colim_{\Delta^{\opposite}} \iota_\setZ F \cong \iota_\setZ(\tot(F)) \in \mcalD(\setZ)
    \end{equation*}
    and we are done.
\end{proof}

The following typical case is used in \Cref{DerivedTensor}.

\begin{corollary} \label{GeometricRealizationisComplex}
    Let \(A\) be a discrete ring and let \(C = (\cdots \to C_i \to \cdots \to C_0 \to 0) \in \Ch_{\geq 0}(A)\) be a non-negatively complex of discrete \(A\)-modules.
    Set a double complex \(C' \defeq (\cdots \to C_i[0] \to \cdots \to C_0[0] \to 0)\) in \(\Ch_{\geq 0}(\Ch(A))\).
    Then the geometric realization of the simplicial object \(\iota_A \DK_{\bullet}(C')\) in \(\mcalD(A)\) is isomorphic to \(\iota_A(C)\) in \(\mcalD(A)\).
    Namely, we have an isomorphism
    \begin{equation*}
        \colim_{\Delta^{\opposite}} \iota_A \DK_{\bullet}(\cdots \to C_i[0] \to \cdots \to C_0[0] \to 0) \cong \iota_A(\cdots \to C_1 \to C_0 \to 0)
    \end{equation*}
    in \(\mcalD(A)\)
\end{corollary}

\begin{proof}
    By \Cref{GeometricRealizationisTotalization}, we have the isomorphism
    \begin{equation*}
        \colim_{\Delta^{\opposite}} \iota_A \DK_{\bullet}(C') \cong \iota_A(\tot(\DK_{\bullet}(C')))
    \end{equation*}
    in \(\mcalD(A)\).
    We show that the totalization \(\tot(\DK_{\bullet}(C'))\) is quasi-isomorphic to the complex \(C = (\cdots \to C_1 \to C_0 \to 0)\) in \(\Ch(A)\).
    Since \(C'_i = C_i[0]\) is concentrated in degree \(0\), \(\DK_{\bullet}(C')([q])_p = (\DK_q(C'))_p = (\oplus_{\alpha \colon [n] \twoheadrightarrow [k]} C_k[0])_p = 0\) holds for all \(p \neq 0\).
    By the construction of \(\tot(-)\) (\Cref{TotSimplicial}), we have
    \begin{equation*}
        \tot(\DK_{\bullet}(C'))_n = \DK_{\bullet}(C')([n])_0 = (\oplus_{\alpha \colon [n] \twoheadrightarrow [k]} C_k[0])_0 = \oplus_{\alpha \colon [n] \twoheadrightarrow [k]} C_k = \DK_n(C)
    \end{equation*}
    in \(\Ch(A)^{\heartsuit}\) and the differential is \(\sum_{i=0}^q (-1)^i \partial_i^{\DK_n(C')_0} = \sum_{i=0}^q (-1)^i \partial_i^{\DK_n(C)}\), where \(\DK_n(C) = \DK_{\bullet}(C)([n])\) is the \(n\)-th term of the simplicial discrete \(A\)-module \(\DK_{\bullet}(C)\) corresponding to \(C \in \Ch_{\geq 0}(A)\).
    So the totalization \(\tot(\DK_{\bullet}(C')) \in \Ch(A)\) is the Moore complex \(M_*(\DK_{\bullet}(C))\) of the simplicial discrete \(A\)-module \(\DK_{\bullet}(C)\).
    The Moore complex \(M_*(\DK_{\bullet}(C))\) (in other words, unnormalized chain complex) and the normalized chain complex \(N_*(\DK_{\bullet}(C))\) are quasi-isomorphic by \cite[Proposition 1.2.3.17]{lurie2017Higher}.
    This shows that the totalization \(\tot(\DK_{\bullet}(C'))\) is quasi-isomorphic to the complex \(C \in \Ch(A)\) because of \(C \cong N_*(\DK_{\bullet}(C))\) in \(\Ch(A)\) (\cite[Lemma 1.2.3.11]{lurie2017Higher}), that is,
    \begin{equation*}
        \colim_{\Delta^{\opposite}} \iota_A \DK_{\bullet}(C') \cong \iota_A(\tot(\DK_{\bullet}(C'))) \cong \iota_A(\tot(C)) \cong \iota_A(C) = \iota_A(\cdots \to C_1 \to C_0 \to 0)
    \end{equation*}
    in \(\mcalD(A)\) where the first isomorphism follows from \Cref{GeometricRealizationisTotalization}.
\end{proof}

In the last part of this section, we record the relation between \(\Ext^i_B(M, N)\) and \(\Ext^i_{\pi_0(B)}(M, N)\) for a certain animated ring \(B\) and its modules \(M\) and \(N\).

\begin{lemma} \label{AssumptionLemma2}
    Let \(A\) be an animated ring and let \(L\) be a connective \(A\)-module.
    Let \(B\) an animated \(A\)-algebra.
    Assume that the connective \(B\)-module \(M \defeq L \otimes^L_A B\) is discrete.
    Then \(\Ext^i_{\pi_0(B)}(M, N)\) is a direct summand of \(\Ext^i_B(M, N)\) for any \(\pi_0(B)\)-module \(N\) and for each \(i \geq 0\).
\end{lemma}

\begin{proof}
    The canonical map of \(A\)-modules \(B \to \pi_0(B)\) induces a commutative diagram
    \begin{center}
        \begin{tikzcd}
            M = L \otimes^L_A B \arrow[r] \arrow[d, "\cong"] & L \otimes^L_A \pi_0(B) \arrow[d] \\
            \pi_0(L \otimes^L_A B) \arrow[r]        & \pi_0(L \otimes^L_A \pi_0(B))   
        \end{tikzcd}
    \end{center}
    in \(\Mod(\pi_0(B))\).
    By \cite[Corollary 7.2.1.23]{lurie2017Higher}, the lower horizontal map is an isomorphism.
    The inverse map gives a splitting of the upper horizontal map and thus \(M = L \otimes^L_A B\) is a direct summand of \(L \otimes^L_A \pi_0(B)\) in \(\Mod(\pi_0(B))\). This shows that the \(\Ext\)-group \(\Ext^i_{\pi_0(B)}(M, N) = \Ext^i_{\pi_0(B)}(L \otimes^L_A B, N)\) is a direct summand of
    \begin{equation*}
        \Ext^i_{\pi_0(B)}(L \otimes^L_A \pi_0(B), N) \cong \Ext^i_{\pi_0(B)}(M \otimes^L_B \pi_0(B), N) \cong \Ext^i_B(M, N)
    \end{equation*}
    for each \(i \geq 0\) and for each \(\pi_0(B)\)-module \(N\).
\end{proof}

In \Cref{AssumptionLemma2}, the vanishing of \(\Ext^i_B(M, N)\) implies the vanishing of its direct factor \(\Ext^i_{\pi_0(B)}(M, N)\).
However, the following example shows that the converse does not hold in general. Remark that in the following example, the ring \(A\) is not regular and the \(A\)-module \(L\) has infinite projective dimension.
If \(L\) has finite projective dimension (or simply \(A\) is regular local), then the converse holds when \(B\) is a derived quotient of \(A\). See \Cref{AssumptionDerived}.

\begin{example} \label{ExtVanishingExample}
    Set \(A \defeq \setZ[x]/(px)\) for some fixed prime number \(p\).
    Let \(L\) be an \(A\)-module \(\setZ\).
    Take an animated \(A\)-algebra \(B \defeq A/^L p = (\setZ[x]/(px))/^L p\) and its module \(M \defeq L \otimes^L_A B \cong \setF_p\).
    Fix a discrete non-zero \(\pi_0(B)\)-module \(N\).
    Then \(M\) is discrete and \(\Ext^i_{\pi_0(B)}(M, N)\) vanishes for each \(i \geq 2\) but not necessarily \(\Ext^i_B(M, N)\) for \(i \geq 2\). Here is the proof.

    As a complex of \(A\)-modules, \(B\) is isomorphic to \(\Kos(p; \setZ[x]/(px))\) and thus
    \begin{equation*}
        M \cong \setZ \otimes^L_{\setZ[x]/(px)} \Kos(p; \setZ[x]/(px)) \cong (0 \to \setZ \xrightarrow{\times p} \setZ \to 0) \cong \setF_p[0]
    \end{equation*}
    is discrete.
    The connected component \(\pi_0(B)\) is isomorphic to \(\setZ[x]/(px, p) \cong \setF_p[x]\).
    The \(\setF_p[x]\)-module \(M = \setF_p\) has a free resolution \(0 \to \setF_p[x] \xrightarrow{\times x} \setF_p[x] \to 0\), which shows that \(\Ext^i_{\pi_0(B)}(M, N) \cong \Ext^i_{\setF_p[x]}(\setF_p, N) = 0\) for each \(i \geq 2\).
    Next, we calculate the \(\Ext\)-group
    \begin{align*}
        \Ext^i_B(M, N) & \cong \Ext^i_{\pi_0(B)}((L \otimes^L_A B) \otimes^L_B \pi_0(B), N) \cong \Ext^i_{\pi_0(B)}(L \otimes^L_A \pi_0(B), N) \\
        & \cong \Ext^i_{\setF_p[x]}(\setZ \otimes^L_{\setZ[x]/(px)} \setF_p[x], N).
    \end{align*}
    The \(\setZ[x]/(px)\)-module \(\setZ\) has a (periodic) free resolution \(\cdots \xrightarrow{\times x} \setZ[x]/(px) \xrightarrow{\times p} \setZ[x]/(px) \xrightarrow{\times x} \setZ[x]/(px) \to 0\).
    This shows that \(\setZ \otimes^L_{\setZ[x]/(px)} \setF_p[x]\) can be represented by the complex \(\cdots \xrightarrow{\times x} \setF_p[x] \xrightarrow{\times p} \setF_p[x] \xrightarrow{\times x} \setF_p[x] \to 0\), which is quasi-isomorphic to the complex of \(\setF_p[x]\)-modules
    \begin{equation*}
        (\cdots \to \setF_p \to 0 \to \setF_p \to 0) \cong \bigoplus_{k \geq 0} \setF_p[2k]
    \end{equation*}
    where \(\setF_p[2k]\) is a complex concentrated in degree \(2k\).
    Taking a free resolution \((0 \to \setF_p[x] \xrightarrow{\times x} \setF_p[x] \to 0) \in \Ch(\setF_p[x])\) of \(\setF_p[2k]\) which is concentrated in degree \(2k+1\) and \(2k\), we have
    \begin{align*}
        \Ext^i_{\setF_p[x]}(\setF_p[2k], N) & = H^i(0 \to \Hom_{\setF_p[x]}(\setF_p[x], N) \xrightarrow{\times x} \Hom_{\setF_p[x]}(\setF_p[x], N) \to 0) \\
        & \cong \begin{cases}
            N/xN & \text{if} \  i = 2k\\
            N[x] & \text{if} \ i = 2k+1\\
            0 & \text{otherwise}
        \end{cases}
    \end{align*}
    through the \(\setF_p[x]\)-module structure on \(N\) where \(N[x]\) is the submodule of \(x\)-torsion elements of \(N\).
    The \(\Ext\)-group \(\Ext^i_{\setF_p[x]}(\oplus_{k \geq 0} \setF_p[2k], N)\) is isomorphic to
    \begin{equation*}
        \Ext^i_{\setF_p[x]}(\oplus_{k \geq 0} \setF_p[2k], N) \cong \bigoplus_{k \geq 0} \Ext^i_{\setF_p[x]}(\setF_p[2k], N) \cong N/xN \  \text{or} \ N[x]
    \end{equation*}
    for each \(i \geq 0\). If \(\setF_p[x]\) acts on \(N\) by \(xN = 0\), this is a desired example (for example \(N = \setF_p\)).
\end{example}

\section{Liftings of Modules}

In this section, we generalize the properties of liftings of modules over usual rings introduced in \cite{auslander1993Liftings} to ones of modules over animated rings.

\begin{definition} \label{DefDerivedLifting}
    Let \(A \to B\) be a map of animated rings and let \(M\) be a \(B\)-module.
    A pair \((L, \varphi)\) of an \(A\)-module \(L\) such that \(\pi_0(L)\) is finite over \(\pi_0(A)\) and an isomorphism \(\varphi \colon M \xrightarrow{\cong} L \otimes^L_A B\) of \(B\)-modules is called a \emph{lifting} of the \(B\)-module \(M\) to \(A\).
    We often denote a lifting \((L, \varphi)\) by \(L\) if the isomorphism \(\varphi\) is clear from the context.
    A \(B\)-module \(M\) is called \emph{liftable} to \(A\) if it has a lifting to \(A\).
\end{definition}

Moreover, we can show the following equivalence of liftings of discrete modules.
Note that any discrete module \(M\) over an animated ring \(A\) is canonically a discrete module over \(\pi_0(A)\).

\begin{lemma} \label{EquivLifting}
    Let \(A \to B\) be a map of animated rings and let \(M\) be a discrete \(B\)-module.
    For a given connective \(A\)-module \(L\), the following are equivalent.
    \begin{enumalphp}
        \item There exists an isomorphism \(\pi_0(L \otimes^L_A B) \cong M\) of discrete \(B\)-modules such that the truncation map \(L \otimes^L_A B \to \pi_0(L \otimes^L_A B) \cong M\) is an isomorphism of \(B\)-modules.
        \item There exists an isomorphism of \(B\)-modules \(L \otimes^L_A B \cong M\), namely, \(L\) (and this isomorphism) is a lifting of the \(B\)-module \(M\) to \(A\).
        \item There exists an isomorphism \(\pi_*(L \otimes^L_A B) \cong \pi_*(M)\) of graded \(\pi_*(B)\)-modules.
    \end{enumalphp}
\end{lemma}

\begin{proof}
    (a) \(\Rightarrow\) (b) \(\Rightarrow\) (c) is clear.
    (c) \(\Rightarrow\) (a) follows from the assumption that \(\pi_i(L \otimes^L_A B)\) vanishes for \(i \neq 0\) and is isomorphic to \(M\) for \(i = 0\).
\end{proof}

We prove the ascending property of the finiteness and discreteness of modules.

\begin{lemma} \label{DiscreteLiftingComp}
    Let \(A\) be an almost perfect animated \(\Lambda\)-algebra and let \(x\) be an element of \(\mfrakm\).
    Fix \(k \geq 1\).
    Let \(M\) be a discrete finitely generated \(A_k\)-module and let \(L\) be an \(A\)-module such that \(\pi_0(L)\) is finite over \(\pi_0(A)\).
    If \(L\) is a lifting of the \(A_k\)-module \(M\) to \(A\), then \(L\) is a discrete finitely generated \(A\)-module.
\end{lemma}

\begin{proof}
    We have a fiber sequence \(L \xrightarrow{\times x^k} L \to L \otimes^L_A A_k \cong M\) and \(\pi_i(L) \xrightarrow{\times x^k} \pi_i(L)\) is surjective for all \(i \neq 0\) since \(M\) is discrete.
    By \Cref{EquivAlmostPerfect} and the finiteness of \(M\), \(M \cong L \otimes^L_A A_k\) is almost perfect over \(A_k\) and thus \(L\) is almost perfect over \(A\) (\Cref{LemBaseChangePerfect} (b)).
    Here, we use the assumption that \(A\) is \(\mfrakm\)-adically complete Noetherian (\Cref{NoetherianAlmostPerfect}) and \(\pi_0(L)\) is finite over \(\pi_0(A)\).
    Since \(\pi_i(L)\) is finite over the Noetherian ring \(\pi_0(A)\), \(\pi_i(L)\) vanishes for all \(i \geq 1\) by Nakayama's lemma and thus \(L\) is discrete.
\end{proof}

\begin{lemma} \label{DiscreteLiftingNil}
    Let \(A\) be an almost perfect animated \(\Lambda\)-algebra and let \(x\) be an element of \(\mfrakm\).
    Fix \(n, k \geq 1\).
    Take the canonical map \(A_{n+k} \to A_n\) of animated \(\Lambda\)-algebras.
    Let \(L_{n+k}\) be an \(A_{n+k}\)-module and let \(L_n\) be a discrete finitely generated \(A_n\)-module.
    If \(L_{n+k}\) is a lifting of the \(A_n\)-module \(L_n\) to \(A_{n+k}\), then \(L_{n+k}\) is a discrete finitely generated \(A_{n+k}\)-module.
\end{lemma}

\begin{proof}
    The base change \(L_{n+k} \otimes^L_{A_{n+k}} A_n \cong L_n\) is discrete and almost perfect over \(A_n\) (here we use the assumption that \(A_n\) is Noetherian and \(L_n = \pi_0(L_n)\) is finite over \(\pi_0(A_n)\)).
    Since \(A_{n+k} \to A_n\) has the nilpotent kernel on \(\pi_0\), the \(A_{n+k}\)-module \(L_{n+k}\) is also connective and almost perfect over \(A_{n+k}\) by \Cref{LemBaseChangePerfectNil}.
    Considering the fiber sequence \(A_k \xrightarrow{\times x^n} A_{n+k} \to A_n\) given by \Cref{DistinguishedTriangleAn} and taking \(L_{n+k} \otimes^L_{A_{n+k}} - \), we have the following fiber sequence in \(\Mod(A_{n+k})\):
    \begin{equation*}
        L_{n+k} \otimes^L_{A_{n+k}} A_k \xrightarrow{\times x^n} L_{n+k} \to L_{n+k} \otimes^L_{A_{n+k}} A_n.
    \end{equation*}
    Since the last term \(L_{n+k} \otimes^L_{A_{n+k}} A_n \cong L_n\) is discrete, the map of \(\pi_0(A_{n+k})\)-modules \(\pi_i(L_{n+k} \otimes^L_{A_{n+k}} A_k) \xrightarrow{\times x^n} \pi_i(L_{n+k})\) is surjective for all \(i \neq 0\).
    Then we have \(x^n \pi_i(L_{n+k}) = \pi_i(L_{n+k})\) and thus \(\pi_i(L_{n+k})\) vanishes for all \(i \neq 0\) by Nakayama's lemma.
\end{proof}




The following theorem is a key point of Auslander, Ding, and Solberg's proof in \cite{auslander1993Liftings} which reduces the lifting problem of general quotients to the one of nilpotent quotients.

\begin{theorem}[{cf. \cite[Theorem 1.2]{auslander1993Liftings}}] \label{LiftingSequence}
    Let \(A\) be an almost perfect animated \(\Lambda\)-algebra and let \(x\) be an element of \(\mfrakm\).
    Take a discrete finitely generated \(A_1\)-module \(M\).
    Then the following are equivalent:
    \begin{enumerate}
        \item \(M\) is liftable to \(A\).
        \item There exists a sequence \(\{L_n\}_{n \geq 1} = \{L_1 \defeq M, L_2, \dots\}\) such that each \(L_n\) is an almost perfect animated \(A_n\)-module and \(L_{n+1}\) is a lifting of \(L_n\) to \(A_{n+1}\) for all \(n \geq 1\).
        \item There exists a sequence \(\{L_n\}_{n \geq 1} = \{L_1 \defeq M, L_2, \dots\}\) such that each \(L_n\) is an \(A_n\)-module and \(L_{n+1}\) is a lifting of \(L_n\) to \(A_{n+1}\) for all \(n \geq 1\).
    \end{enumerate}
    In this case, any such \(A_n\)-module \(L_n\) and any lifting \(L\) of \(M\) to \(A\) are discrete and finitely generated.
\end{theorem}

\begin{proof}
    Since \(M\) is discrete, each assumption is followed by the last assertion (\Cref{DiscreteLiftingComp} and \Cref{DiscreteLiftingNil}).

    \IMPLIES{(1)}{(2)}: Since \(M\) has a lifting \(L\) to \(A\), we can set an \(A_n\)-module \(L_n \defeq L/^L x^n \cong L \otimes^L_A A_n\) for each \(n \geq 1\).
    Since \(M\) is almost perfect over \(A_1\) and \(L_n \otimes^L_{A_n} A_1 \cong M\), \(L_n\) is almost perfect over \(A_n\) by \Cref{LemBaseChangePerfectNil}.
    The isomorphisms of \(A_n\)-modules
    \begin{equation*}
        L_{n+1} \otimes^L_{A_{n+1}} A_n \cong (L \otimes^L_A A_{n+1}) \otimes^L_{A_{n+1}} A_n \cong L \otimes^L_A A_n \cong L_n.
    \end{equation*}
    show that the \(A_{n+1}\)-module \(L_{n+1}\) is a lifting of \(L_n\) to \(A_{n+1}\).

    \IMPLIES{(2)}{(3)}: This is clear.

    \IMPLIES{(3)}{(1)}: By using \Cref{DiscreteLiftingNil} inductively, each \(L_n\) is an almost perfect animated \(A_n\)-module.
    Applying \(L_{n+1} \otimes^L_{A_{n+1}} -\) for the fiber sequence \(A_n \xrightarrow{\times x} A_{n+1} \to A_1\) in \Cref{DistinguishedTriangleAn}, we have a fiber sequence in \(\Mod(A_{n+1})\)
    \begin{equation} \label{IteratedFiberSequence}
        L_{n+1} \otimes^L_{A_{n+1}} A_n \xrightarrow{\id_{L_n} \otimes \times x} L_{n+1} \otimes^L_{A_{n+1}} A_{n+1} \to L_{n+1} \otimes^L_{A_{n+1}} A_1.
    \end{equation}
    Furthermore, the assumption is followed by the following isomorphisms in \(\Mod(A_{n+1})\):
    \begin{align*}
        L_{n+1} \otimes^L_{A_{n+1}} A_n & \cong L_n, \\
        L_{n+1} \otimes^L_{A_{n+1}} A_{n+1} & \cong L_{n+1}, \\
        L_{n+1} \otimes^L_{A_{n+1}} A_1 & \cong L_{n+1} \otimes^L_{A_{n+1}} A_n \otimes^L_{A_n} A_1 \cong L_n \otimes^L_{A_n} A_1 \cong \dots \cong L_2 \otimes^L_{A_2} A_1 \cong M.
    \end{align*}
    Then the above fiber sequence (\ref{IteratedFiberSequence}) induces a fiber sequence \(L_n \xrightarrow{\times x} L_{n+1} \to M\) in \(\Mod(A_{n+1})\) and thus we have an exact sequence \(0 \to L_n \xrightarrow{\times x} L_{n+1} \to M \to 0\) of discrete \(\pi_0(A_{n+1})\)-modues.
    Since the canonical surjective map \(L_{n+1} \twoheadrightarrow L_{n+1} \otimes^L_{A_{n+1}} A_n \cong L_n\) gives an inverse system \(\{L_n\}_{n \geq 1}\) of discrete \(\pi_0(A)\)-modules, we have an exact sequence of inverse systems of discrete \(\pi_0(A)\)-modules
    \begin{equation*}
        0 \to \{L_n\}_{n \geq 1} \xrightarrow{\times x} \{L_{n+1}\}_{n \geq 1} \to \{M\}_{n \geq 1} \to 0.
    \end{equation*}
    Taking the limit, \(L \defeq \lim_{n \geq 1} L_n\) is an \(\pi_0(A) \cong \lim_{n \geq 1} \pi_0(A_n)\)-module since \(\pi_0(A)\) is \(\mfrakm\)-adically complete (\Cref{NoetherianAlmostPerfect}).
    So we have an exact sequence
    \begin{equation*}
        0 \to L \xrightarrow{\times x} L \to M \to 0
    \end{equation*}
    of discrete \(\pi_0(A)\)-modules.
    This induces an isomorphism \(L \otimes^L_A A_1 \cong L/^L x \cong M\) as \(A_1\)-modules because of \(L/xL \cong M\) and \(L[x] = 0\).
    To prove \(L\) is a lifting of \(M\) to \(A\), we must show that \(\pi_0(L) = L\) is finite over \(\pi_0(A)\).
    As explained at the begining of this proof, each \(L_n\) is almost perfect and thus a discrete finitely generated \(\pi_0(A_n)\)-module by \Cref{EquivAlmostPerfect}.
    Applying the above argument replacing \(M\) with \(L_n\) and \(x\) with \(x^n\), we have an exact sequence \(0 \to L \xrightarrow{\times x^n} L \to L_n \to 0\) and thus \(L_n \cong L/x^nL\).
    As in the same proof of \cite[Theorem 1.2]{auslander1993Liftings}, \(L\) is a discrete finitely generated \(\pi_0(A)\)-module.
    %
    %
\end{proof}



\begin{lemma}[{cf. \cite[Lemma 1.4]{auslander1993Liftings}}] \label{LiftingEquivFiberSeq}
    Let \(A\) be an almost perfect animated \(\Lambda\)-algebra and let \(x\) be an element of \(\mfrakm\).
    Take a discrete finitely generated \(A_1\)-module \(M\) and its lifting \(L\) to \(A_n\).
    For an \(A_{n+1}\)-module \(E\), the following are equivalent:
    \begin{enumerate}
        \item \(E\) is a lifting of the \(A_n\)-module \(L\) to \(A_{n+1}\).
        \item There exists a fiber sequence in \(\Mod(A_{n+1})\) of the form \(M \to E \to L\) and there exists an isomorphism \(\pi_0(E \otimes^L_{A_{n+1}} A_1) \cong M\) of \(\pi_0(A_1)\)-modules.
    \end{enumerate}
    In this case, the \(A_{n+1}\)-module \(E\) is a discrete finitely generated \(\pi_0(A_{n+1})\)-module by \Cref{DiscreteLiftingNil}.
\end{lemma}

\begin{proof}
    By \Cref{DiscreteLiftingNil}, \(L\) is a discrete finitely generated \(\pi_0(A_n)\)-module.

    \IMPLIES{(1)}{(2)}:
    First note that we have a fiber sequence in \(\Mod(A_{n+1})\):
    \begin{equation*}
        A_1 \xrightarrow{\times x^n} A_{n+1} \to A_n
    \end{equation*}
    by \Cref{DistinguishedTriangleAn}.
    Taking base change \(E \otimes^L_{A_{n+1}} -\), we have a fiber sequence in \(\Mod(A_{n+1})\)
    \begin{equation*}
        E \otimes^L_{A_{n+1}} A_1 \xrightarrow{\id_L \otimes \times x^n} E \otimes^L_{A_{n+1}} A_{n+1} \to E \otimes^L_{A_{n+1}} A_n.
    \end{equation*}
    Since \(E\) is a lifting of \(L\) to \(A_{n+1}\) and \(L\) is a lifting of \(M\) to \(A_n\), the following isomorphisms hold in \(\Mod(A_{n+1})\):
    \begin{align*}
        E \otimes^L_{A_{n+1}} A_1 & \cong (E \otimes^L_{A_{n+1}} A_n) \otimes^L_{A_n} A_1 \cong L \otimes^L_{A_n} A_1 \cong M \\
        E \otimes^L_{A_{n+1}} A_{n+1} & \cong E \\
        E \otimes^L_{A_{n+1}} A_n & \cong L.
    \end{align*}
    By these isomorphisms, the desired isomorphism \(\pi_0(E \otimes^L_{A_{n+1}} A_1) \cong M\) and the fiber sequence in \(\Mod(A_{n+1})\) exist.

    \IMPLIES{(2)}{(1)}: Assume that there exists a fiber sequence in \(\Mod(A_{n+1})\) of the form \(M \to E \to L\).
    Since \(M\) and \(L\) are discrete, so is \(E\) and thus there exists an exact sequence \(0 \to M \to E \to L \to 0\) of finite discrete \(\pi_0(A_{n+1})\)-modules.
    By \Cref{EquivLifting}, it suffices to show the isomorphism \(\pi_*(E \otimes^L_{A_{n+1}} A_n) \cong \pi_*(L)\) holds.
    The assumption that \(L\) is a lifting of \(M\) to \(A_n\) shows that
    \begin{equation*}
        \pi_k(L \otimes^L_{A_n} A_1) \cong 
        \begin{cases}
            L/xL \cong M & \text{if } k = 0 \\
            L[x]/x^{n-1}L \cong 0 & \text{if \(k\) is odd} \\
            L[x^{n-1}]/xL \cong 0 & \text{if \(k\) is even}
        \end{cases}
    \end{equation*}
    by \Cref{DerivedTensor}.
    Also, we have
    \begin{equation*}
        \pi_k(E \otimes^L_{A_{n+1}} A_n) \cong 
        \begin{cases}
            E/x^nL & \text{if } k = 0 \\
            E[x^n]/xE & \text{if \(k\) is odd} \\
            E[x]/x^nE & \text{if \(k\) is even}.
        \end{cases}
    \end{equation*}
    We must show that \(E/x^nE \cong L\) and \(\pi_k(E \otimes^L_{A_{n+1}} A_n) = 0\) for \(k > 0\) but this can be shown by the same proof of \cite[Lemma 1.4]{auslander1993Liftings} since \(\pi_0(A)\) is a Noetherian ring and \(M\) is a discrete Noetherian \(\pi_0(A)\)-module.

\end{proof}


\section{Fiber Sequence Corresponding to a Lifting}

In this section, we construct a fiber sequence \(\theta'_L\) of \(A_1\)-modules corresponding to a lifting.
This construction is based on the one for the case of complete intersections in \cite{auslander1993Liftings}.
The existence of such a fiber sequence is crucial in the proof of \Cref{LiftingCorollary} and one of the advantages of using higher algebra is that it can handle the derived category \(\Mod(A_n)\) of a non-discrete ring \(A_n\) in this manner.

\begin{construction} \label{LiftingExtElement}
    Let \(A\) be an almost perfect animated \(\Lambda\)-algebra and let \(x\) be an element of \(\mfrakm\).
    Fix \(n \geq 1\).
    Take a discrete finitely generated \(A_1\)-module \(M\) and its lifting \(L\) to \(A_n\) (in this case, \(L\) is a discrete finitely generated \(A_n\)-module by \Cref{DiscreteLiftingNil}).
    We define two fiber sequences \(\theta_L\) (\ref{ThetaL}) and \(\theta_L'\) (\ref{ThetaL'}) in \(\Mod(A_1)\) as follows.

    Choose a map of \(A\)-modules \(\map{p}{P}{L}\), where \(P\) is a finite free \(A\)-module and \(\map{\pi_0(p)}{\pi_0(P)}{L}\) is surjective as in the proof of \cite[Proposition 7.2.2.7]{lurie2017Higher} (for example, \(P \defeq \oplus_{x \in S} A\) where \(S \subseteq L\) is a system of generators of \(L\) over \(\pi_0(A_n)\)).
    Taking its fiber, we have a fiber sequence of \(A\)-modules
    \begin{equation}
        \Omega_A(L) \defeq \fib(p) \to P \xrightarrow{p} L. \label{FibL}
    \end{equation}
    Applying the base change functor \(- \otimes^L_A A_1 \cong -/^L x\) (see \Cref{LemDerivedQuotient}), we can take a fiber sequence \(\theta_L\) of \(A_1\)-modules
    \begin{equation}
        \theta_L \colon \Omega_A(L)/^L x \to P/^L x \xrightarrow{p/^L x} L/^L x. \label{ThetaL}
    \end{equation}
    In particular, this \(\theta_L\) is a distinguished triangle in the homotopy category \(\mathrm{h}\Mod(A_1)\).

    Next, we construct a fiber sequence \(\theta_L'\) in \(\Mod(A_1)\) from \(\theta_L\).
    In the homotopy category \(\mathrm{h}\Mod(\pi_0(A_n))\) of \(\pi_0(A_n)\)-modules, which is equivalent to the derived category \(D(\pi_0(A_n))\), we have a sequence of complexes:
    \begin{center}
        \begin{tikzcd}
            0 \arrow[r] & L \arrow[r, "\times x"] \arrow[d, Rightarrow, no head] & L \arrow[r] \arrow[d, two heads]              & 0 \\
            0 \arrow[r] & L \arrow[r, "\times x"]                                & L/xL \arrow[r]                                & 0 \\
            0 \arrow[r] & L/xL \arrow[r, "\times x"] \arrow[u, "\times x^{n-1}"] & L/xL. \arrow[r] \arrow[u, Rightarrow, no head] & 0
        \end{tikzcd}
    \end{center}
    Since \(L\) is a lifting of \(M\) to \(A_n\), the homotopy group \(\pi_i(L \otimes^L_{A_n} A_1)\) vanishes for \(i \neq 0\) and \(L\) is discrete.
    By \Cref{DerivedTensor}, we have \(L[x] = x^{n-1}L\) and \(L[x^{n-1}] = xL\), and in particular, the map of \(\pi_0(A_1)\)-modules
    \begin{equation*}
        L/xL \xrightarrow{\times x^{n-1}} x^{n-1}L = L[x] = \pi_1(0 \to L \xrightarrow{\times x} L \to 0)
    \end{equation*}
    is isomorphism.
    So the above diagram gives isomorphisms on homotopy groups;
    \begin{equation*}
        L[x] \xrightarrow{\id} L[x] \xleftarrow{\times x^{n-1}} L/xL~\text{on}~\pi_0~\text{and}~L/xL \xrightarrow{\id} L/xL \xleftarrow{\id} L/xL~\text{on}~\pi_1.
    \end{equation*}
    This shows that the maps \(L/^L x \to (0 \to L \xrightarrow{\times x} L/xL \to 0) \leftarrow L/xL[1] \oplus L/xL\) are isomorphisms in \(\mathrm{h}\Mod(\pi_0(A_n))\).
    This gives rise to isomorphisms
    \begin{equation} \label{LxLCoproduct}
        L/^L x \cong L/xL[1] \oplus L/xL
    \end{equation}
    in \(\Mod(A_n)\).
    Passing to \(\Mod(A_n)\), we have an isomorphism \(L/^L x \cong L/xL[1] \oplus L/xL\) in \(\Mod(A_n)\) (and thus in \(\Mod(A_1)\)).

    The surjective map of \(A_1\)-modules \(L/^L x \twoheadrightarrow \pi_0(L/^L x) \cong M\) gives the composition map \(q \colon P/^L x \xrightarrow{p/^L x} L/^L x \twoheadrightarrow M\) where the map \(\map{p/^L x}{P/^L x}{L/^L x}\) is the map in (\ref{ThetaL}).
    Taking its fiber in \(\Mod(A_1)\)
    \begin{equation} \label{SyzygyM}
        \Omega_{A_1}(M) \defeq \fib(q) \to P/^L x \xrightarrow{q} M.
    \end{equation}
    Considering the universal property of \(\Omega_{A_1}(M) = \fib(q)\), the following commutative diagram whose horizontal sequences are fiber sequences in \(\Mod(A_1)\) exists:
    \begin{equation} \label{DiagramFiberSequence}
        \begin{tikzcd}
            \Omega_A(L)/^L x \arrow[r] \arrow[d, "\exists ! r", dashed] & P/^L x \arrow[r, "p/^L x"] \arrow[d, Rightarrow, no head] & L/^L x \arrow[d, two heads] \\
            \Omega_{A_1}(M) \arrow[r]                                 & P/^L x \arrow[r, "q"]                                     & {M.}                     
        \end{tikzcd}
    \end{equation}
    Taking the long exact sequences of homotopy groups, we have the following commutative diagram in the category of discrete \(\pi_0(A_1)\)-modules for each \(i \in \setZ\):
    \begin{equation} \label{LongExactDiagram}
        \begin{tikzcd}
            \pi_{i+1}(P/^L x) \arrow[r, "\pi_{i+1}(p/^Lx)"] \arrow[d, Rightarrow, no head] & \pi_{i+1}(L/^L x) \arrow[r] \arrow[d, two heads] & \pi_i(\Omega_A(L)/^L x) \arrow[r] \arrow[d, "\pi_i(r)"] & \pi_i(P/^L x) \arrow[r, "\pi_i(p/^Lx)"] \arrow[d, Rightarrow, no head] & \pi_i(L/^L x) \arrow[d, two heads] \\
            \pi_{i+1}(P/^L x) \arrow[r, "\pi_{i+1}(q)"]                                    & \pi_{i+1}(M) \arrow[r]                           & \pi_i(\Omega_{A_1}(M)) \arrow[r]                        & \pi_i(P/^L x) \arrow[r, "\pi_i(q)"]                                    & \pi_i(M)                          
        \end{tikzcd}
    \end{equation}
    In particular, we can show the following results:
    \begin{enumalphp}
        \item If \(i \geq 2\), then \(\pi_i(L/^L x)\) and \(\pi_i(M)\) vanish and thus \(\map{\pi_i(r)}{\pi_i(\Omega_A(L)/^L x)}{\pi_i(\Omega_{A_1}(M))}\) is an isomorphism.
        \item The vanishing \(\pi_2(L/^L x) = 0\) implies that the map \(\pi_1(r)\) is injective.
        \item The vanishing \(\pi_1(M) = 0\) and the isomorphy \(\pi_0(L/^L x) \cong \pi_0(M)\) imply that the sequence \(\pi_1(L/^L x) \to \pi_0(\Omega_A(L)/^L x) \xrightarrow{\pi_0(r)} \pi_0(\Omega_{A_1}(M)) \to 0\) is exact.
        \item If \(i < 0\), the connectivity implies that \(\pi_i(r)\) is an isomorphism.
    \end{enumalphp}
    By the fiber sequence \(L/^L x[-1] \to \Omega_A(L)/^L x \to P/^L x\) given by \(\theta_L\) (\ref{ThetaL}), composing with \(r\) gives maps of \(A_1\)-modules
    \begin{equation} \label{DefAlpha}
        M \cong \tau_{\geq 0}(L/^L x[-1]) \to L/^L x[-1] \to \Omega_A(L)/^L x \xrightarrow{r} \Omega_{A_1}(M)
    \end{equation}
    where the first isomorphism follows from (\ref{LxLCoproduct}).
    This composition \(M \to \Omega_{A_1}(M)\) is zero. Set the composition of the first two maps to be \(\map{\alpha}{M}{\Omega_A(L)/^L x}\) in \(\Mod(A_1)\).
    Furthermore, the commutative diagram (\ref{DiagramFiberSequence}) provides \(\Omega_{A_1}(M) \to P/^L x \xrightarrow{p/^L x} L/^L x \cong L/xL[1] \oplus L/xL \twoheadrightarrow L/xL[1] \cong M[1]\) in \(\Mod(A_1)\). We have a sequence of maps
    \begin{equation}
        M \xrightarrow{\alpha} \Omega_A(L)/^L x \xrightarrow{r} \Omega_{A_1}(M) \to M[1] \label{DistinguishedTriangleR}
    \end{equation}
    in \(\Mod(A_1)\) such that the composition of any two successive maps is zero.
    In the following, we need to show that the sequence \(M \xrightarrow{\alpha} \Omega_A(L)/^L x \xrightarrow{r} \Omega_{A_1}(M)\) in \(\Mod(A_1)\) is a fiber sequence.
    By \Cref{TriangleDistinguishedTriangle} below and the discreteness of \(M\), it suffices to show that the homotopy long sequence induced from \(M \xrightarrow{\alpha} \Omega_A(L)/^L x \xrightarrow{r} \Omega_{A_1}(M) \to M[1]\) is exact.

    By (a), (b), (c), and (d) above, \(\pi_i(M) \to \pi_i(\Omega_A(L)/^L x) \xrightarrow{\pi_i(r)} \pi_i(\Omega_{A_1}(M))\) is exact for \(i \in \setZ\).

    We next show the exactness of \(\pi_i(\Omega_A(L)/^L x) \xrightarrow{\pi_i(r)} \pi_i(\Omega_{A_1}(M)) \to \pi_{i-1}(M)\) for \(i \in \setZ\).
    If \(i \geq 2\), the isomorphy of \(\pi_i(r)\) by (a) and the vanishing of \(\pi_{i-1}(M)\) shows the exactness.
    If \(i = 0\), the surjectivity of \(\pi_0(r)\) by (c) shows the exactness.
    It remains to show the exactness for \(i = 1\), that is, if an element \(a \in \pi_1(\Omega_{A_1}(M))\) maps to \(0\) via
    \begin{equation*}
        \pi_1(\Omega_{A_1}(M)) \to \pi_1(P/^L x) \xrightarrow{\pi_1(p/^L x)} \pi_1(L/^L x) \cong \pi_1(L/xL[1] \oplus L/xL) \xrightarrow{\cong} M
    \end{equation*}
    then it belongs to the image of \(\pi_1(r)\).
    The first map \(\pi_1(\Omega_{A_1}(M)) \to \pi_1(P/^L x)\) is an isomorphism by (\ref{SyzygyM}) and the discreteness of \(M\).
    By (\ref{DiagramFiberSequence}) and (\ref{LongExactDiagram}), we have the following commutative diagram:
    \begin{center}
        \begin{tikzcd}
            \pi_1(\Omega_A(L)/^L x) \arrow[r] \arrow[d, "\pi_1(r)"'] & \pi_1(P/^L x) \arrow[rr, "\pi_1(p/^L x)"] &  & {\pi_1(L/^L x) \cong \pi_1(L/xL[1] \oplus L/xL)} \\
            \pi_1(\Omega_{A_1}(M)) \arrow[ru, "\cong"] \arrow[rrru]  &                                           &  &                                                 
        \end{tikzcd}
    \end{center}
    such that the upper horizontal sequence is exact. This shows the desired claim.


    Next, we show the exactness of the sequence \(\pi_{i+1}(\Omega_{A_1}(M)) \to \pi_i(M) \xrightarrow{\pi_i(\alpha)} \pi_i(\Omega_A(L)/^L x)\).
    By the discreteness of \(M\), it suffices to show the case of \(i = 0\).
    The first map is
    \begin{equation} \label{A1syzygyPi1}
        \pi_1(\Omega_{A_1}(M)) \to \pi_1(P/^L x) \xrightarrow{\pi_1(p/^L x)} \pi_1(L/^L x) \cong M
    \end{equation}
    and the second map is
    \begin{equation*}
        \pi_0(\alpha) \colon M \cong \pi_1(L/^L x) \to \pi_0(\Omega_A(L)/^L x)
    \end{equation*}
    by (\ref{DefAlpha}).
    If \(\pi_0(\alpha)\) maps \(a \in M\) to \(0\) in \(\pi_0(\Omega_A(L)/^L x)\), there exists an element \(b \in \pi_1(P/^L x)\) such that \(\pi_1(p/^L x)(b) = a \in \pi_0(M)\) by the exactness of the upper vertical sequence in (\ref{LongExactDiagram}).
    This \(b\) maps to \(0\) via \(\map{\pi_1(q)}{\pi_1(P/^L x)}{\pi_1(M) = 0}\) and we can take \(c \in \pi_1(\Omega_{A_1}(M))\) which goes to \(b\) in \(\pi_1(P/^L x)\) by the lower exact sequence in (\ref{LongExactDiagram}).
    This \(c\) maps to \(a\) in \(\pi_1(L/^L x) \cong M\) via the first map of (\ref{A1syzygyPi1}). This shows the desired exactness.

    Therefore, the sequence (\ref{DistinguishedTriangleR}) in \(\Mod(A_1)\) gives a fiber sequence
    \begin{equation}
        \theta_L' \colon M \xrightarrow{\alpha} \Omega_A(L)/^L x \xrightarrow{r} \Omega_{A_1}(M) \label{ThetaL'}
    \end{equation}
    in \(\Mod(A_1)\) and an exact sequence
    \begin{equation}
        \pi_0(\theta_L') \colon M \xrightarrow{\alpha_0} \pi_0(\Omega_A(L)/^L x) \xrightarrow{r_0} \pi_0(\Omega_{A_1}(M)) \to 0 \label{ExactSequenceThetaL'}
    \end{equation}
    of discrete \(\pi_0(A_1)\)-modules.
    The two fiber sequences \(\theta_L\) and \(\theta_L'\) give the commutative diagram whose vertical maps are distinguished triangles in \(\mathrm{h}\Mod(A_1)\):
    \begin{equation} \label{DiagramThetaLThetaL'}
        \begin{tikzcd}
            & {L/^L x[-1]} \arrow[r, "-1"]     & \Omega_A(L)/^L x \arrow[r] \arrow[d, Rightarrow, no head] & P/^L x \arrow[r]                          & L/^L x \arrow[d, two heads] \arrow[r, "+1"] & {} \\
            {} \arrow[r, "-1"] & {M} \arrow[r] \arrow[u] & \Omega_A(L)/^L x \arrow[r, "r"]                           & \Omega_{A_1}(M) \arrow[r, "+1"] \arrow[u] & {M[1],}                                      &   
        \end{tikzcd}
    \end{equation}
    the upper sequence is \(\theta_L\) and the lower one is \(\theta'_L\).
\end{construction}

In the above construction, we use the following lemma.

\begin{lemma} \label{TriangleDistinguishedTriangle}
    Let \(\mcalC\) be a stable \(\infty\)-category with a \(t\)-structure (see \cite[Definition 1.2.11]{lurie2017Higher} for some notation).
    For example, \(\mcalC = \Mod(A)\) by \cite[Remark 1.4.3.8]{lurie2017Higher}.
    Let \(L \xrightarrow{g} M \xrightarrow{f} N \xrightarrow{h} L[1]\) be a sequence of maps in \(\mcalC\) such that any two successive maps are zero in \(\mcalC\).
    By the universality of \(\fib(f)\), we can take a map \(\map{s}{L}{\fib(f)}\) in \(\mcalC\) such that the following diagram commutes:
    \begin{center}
        \begin{tikzcd}
            L \arrow[r, "g"] \arrow[d, "s"', dashed] & M \arrow[r, "f"] \arrow[d, Rightarrow, no head] & N \arrow[d, Rightarrow, no head] \\
            \fib(f) \arrow[r]                         & M \arrow[r, "f"]                                  & N.                              
        \end{tikzcd}
    \end{center}
    Assume that \(M\) and \(N\) are connective and \(L\) is discrete.
    Then the following are equivalent:
    \begin{enumerate}
        \item The map \(L \xrightarrow{s} \fib(f)\) is an isomorphism in \(\mcalC\). In particular, \(L \xrightarrow{g} M \xrightarrow{f} N\) is a fiber sequence in \(\mcalC\).
        \item The homotopy long sequence \(\cdots \to \pi_i(L) \xrightarrow{\pi_i(g)} \pi_i(M) \xrightarrow{\pi_i(f)} \pi_i(N) \to \cdots\) in \(\mcalC^{\heartsuit}\) is an exact sequence.
    \end{enumerate}
\end{lemma}

\begin{proof}
    (1) \(\Rightarrow\) (2): This follows from a general fact about \(t\)-structures.

    (2) \(\Rightarrow\) (1): By the discreteness of \(L\), \(\map{\pi_i(f)}{\pi_i(M)}{\pi_i(N)}\) is an isomorphism for \(i \neq 0\) and surjective for \(i = 0\).
    Considering the long exact sequence of the fiber sequence \(\fib(f) \to M \xrightarrow{f} N\), we can show that \(\pi_i(\fib(f)) = 0\) for \(i \neq 0\). The fiber \(\fib(f)\) is discrete.
    The connectivity of \(M\) and \(N\) implies the commutative diagram whose horizontal sequences are exact
    \begin{center}
        \begin{tikzcd}
            0 \arrow[r] \arrow[d, Rightarrow, no head] & \pi_0(L) \arrow[r, "\pi_0(g)"] \arrow[d, "\pi_0(s)"] & \pi_0(M) \arrow[r, "\pi_0(f)"] \arrow[d, Rightarrow, no head] & \pi_0(N) \arrow[d, Rightarrow, no head]\\
            0 \arrow[r]                              & \pi_0(\fib(f)) \arrow[r]                       & \pi_i(M) \arrow[r, "\pi_0(f)"]                       & \pi_0(N).
        \end{tikzcd}
    \end{center}
    The five lemma implies that \(\pi_0(s)\) is an isomorphism and we are done.
\end{proof}

\section{\texorpdfstring{\(\Ext\)}{Ext}-vanishings and Liftings of Modules}

In this section, we prove our main theorem (\Cref{LiftingCorollary}). The following argument is a ``derived'' generalization of \cite{auslander1993Liftings}.
We fix an integer \(n \geq 1\).



\begin{lemma} \label{LiftingSplits}
    Let \(A\) be an almost perfect animated \(\Lambda\)-algebra and let \(x\) be an element of \(\mfrakm\).
    Take a discrete finitely generated \(A_1\)-module \(M\) and its lifting \(L\) to \(A_n\).
    We can take a fiber sequence \(\theta_L'\) in \(\Mod(A_1)\) and an exact sequence \(\pi_0(\theta_L')\) in \(\Mod(A_1)^{\heartsuit}\) by (\ref{ThetaL'}) and (\ref{ExactSequenceThetaL'}).
    Then the following are equivalent:
    \begin{enumalphp}
        \item The map \(\map{\pi_0(\alpha)}{M}{\pi_0(\Omega_A(L)/^L x)}\) is a split injection in the category \(\Mod(A_1)^{\heartsuit}\) of discrete \(\pi_0(A_1)\)-modules.
        \item The fiber sequence \(\theta_L'\) splits in \(\mathrm{h}\Mod(A_1)\).
        \item The map \(\map{\alpha}{M}{\Omega_A(L)/^L x}\) induces an isomorphism \(\Omega_A(L)/^L x \cong M \oplus \Omega_{A_1}(M)\) in \(\mathrm{h}\Mod(A_1)\).
        \item The \(A_n\)-module \(L\) is liftable to \(A_{n+1}\).
    \end{enumalphp}
    In this case, any lifting \(E\) of the \(A_n\)-module \(L\) to \(A_{n+1}\) is a discrete finitely generated \(A_{n+1}\)-module by \Cref{DiscreteLiftingNil}.
\end{lemma}

\begin{proof}
    Note that \(L\) is a discrete finitely generated \(A_n\)-module by \Cref{DiscreteLiftingNil}.
    (b) \(\Leftrightarrow\) (c) and (b) \(\Rightarrow\) (a) are clear.
    If (a) holds, we can take a map \(\map{\beta_0}{\pi_0(\Omega_A(L)/^L x)}{M}\) of discrete \(\pi_0(A_1)\)-modules such that the composition \(M \xrightarrow{\pi_0(\alpha)} \pi_0(\Omega_A(L)/^L x) \xrightarrow{\beta_0} M\) is the identity map.
    Since \(M\) is discrete, \(\pi_0(\alpha)\) can be written as \(M \xrightarrow{\alpha} \Omega_A(L)/^L x \to \pi_0(\Omega_A(L)/^L x)\).
    Then the composition \(\Omega_A(L)/^L x \to \pi_0(\Omega_A(L)/^L x) \xrightarrow{\beta_0} M\) gives the splitting of \(\alpha\) and thus (b) holds.

    We prove (b) \(\Rightarrow\) (d).
    Since \(\theta_L'\) splits in \(\mathrm{h}\Mod(A_1)\), we have a map \(\map{\beta}{\Omega_A(L)/^L x}{M}\) in \(\Mod(A_1)\) such that the composition \(M \xrightarrow{\alpha} \Omega_A(L)/^L x \xrightarrow{\beta} M\) is (homotopically equivalent to) the identity map in \(\Mod(A_1)\).
    Set the map \(\gamma \colon \Omega_A(L) \to \Omega_A(L)/^L x \xrightarrow{\beta} M\) in \(\Mod(A)\).
    Then there exists an \(A_n\)-module \(E\) and the following commutative diagram in \(\Mod(A)\)
    \begin{equation} \label{DiagramSplit}
        \begin{tikzcd}
            \Omega_A(L) \arrow[r] \arrow[d, "\gamma"] & P \arrow[r] \arrow[d] & L \arrow[d, Rightarrow, no head] \\
            {M} \arrow[r, "g"]                          & E \arrow[r, "f"]   & L                               
        \end{tikzcd}
    \end{equation}
    whose upper and lower horizontal sequences are fiber sequences in \(\Mod(A)\).
    In fact, \(E\) can be taken as the cofiber of the composition \(L[-1] \to \Omega_A(L) \xrightarrow{\gamma} M\) in \(\Mod(A_n)\) and we use the fact that the restriction of scalar preserves small limits and colimits (\cite[Corollary 4.2.3.7]{lurie2017Higher}).
    Since \(L\) is discrete, so is \(E\).
    By \Cref{LiftingEquivFiberSeq}, we want to show that there exists an isomorphism \(\pi_0(E \otimes^L_{A_{n+1}} A_1) \cong M\) of \(\pi_0(A_1)\)-modules.

    Taking derived quotient \(-/^L x\) of the above diagram, we have a following commutative diagram in \(\Mod(A_1)\);
    \begin{equation*}
        \begin{tikzcd}
            &  & M \arrow[d, "\alpha"] \arrow[lld, hook]                                      &                                                  &                             &                                        \\
            {L/^L x[-1]} \arrow[dd, Rightarrow, no head] \arrow[rr] &  & \Omega_A(L)/^L x \arrow[rr] \arrow[dd, "\gamma/^L x"'] \arrow[rd, two heads] &                                                  & P/^L x \arrow[r] \arrow[dd] & L/^L x \arrow[dd, Rightarrow, no head] \\
            &  &                                                                              & (\Omega_A(L)/^L x)/^L x \arrow[ld, "\beta/^L x"'] &                             &                                        \\
            {L/^L x[-1]} \arrow[rr]                                 &  & M/^L x \arrow[rr, "g/^L x"]                                                       &                                                  & E/^L x \arrow[r, "f/^L x"]       & L/^L x.                                
        \end{tikzcd}
    \end{equation*}
    Taking long exact sequences, we have a following commutative diagram of discrete \(\pi_0(A_1)\)-modules;
    \begin{equation} \label{DiscreteDiagram}
        \begin{tikzcd}[column sep=small]
            &  & M \arrow[d, "\pi_0(\alpha)"] \arrow[lld, "\cong"']                                       &                                                                 &                                    &                                                &                                   \\
            L/xL \arrow[dd, Rightarrow, no head] \arrow[rr] &  & \pi_0(\Omega_A(L)/^L x) \arrow[rr] \arrow[dd, "\pi_0(\gamma/^L x)"'] \arrow[rd, "\cong"] &                                                                 & \pi_0(P/^L x) \arrow[r] \arrow[dd] & L/xL \arrow[dd, Rightarrow, no head] \arrow[r] & 0 \arrow[dd, Rightarrow, no head] \\
            &  &                                                                                          & \pi_0((\Omega_A(L)/^L x)/^L x) \arrow[ld, "\pi_0(\beta/^L x)"'] &                                    &                                                &                                   \\
            L/xL \arrow[rr]                                 &  & M \arrow[rr, "\pi_0(g/^L x)"]                                                            &                                                                 & E/xE \arrow[r, "\pi_0(f/^L x)"]    & L/xL. \arrow[r]                                & 0                                
        \end{tikzcd}
    \end{equation}
    The isomorphism \(\pi_0(\Omega_A(L)/^L x) \cong \pi_0((\Omega_A(L)/^L x)/^L x)\) is the identity map and thus \(\pi_0(\beta/^L x) = \pi_0(\beta)\). So the map \(M \cong L/xL \to \pi_0(\Omega_A(L)/^L x) \xrightarrow{\pi_0(\gamma/^L x)} M\) is equal to \(\pi_0(\beta) \circ \pi_0(\alpha) = \id\).
    In particular, the left lower map \(L/xL \to M\) is surjective and thus \(\map{\pi_0(f/^L x)}{E/xE}{L/xL}\) is an isomorphism.

    We next prove (d) \(\Rightarrow\) (a).
    Let \(E\) be a lifting of \(L\) to \(A_{n+1}\). By \Cref{DiscreteLiftingNil} and \Cref{LiftingEquivFiberSeq}, \(E\) is a discrete finitely generated \(A_{n+1}\)-module and there exists a fiber sequence \(M \xrightarrow{g} E \xrightarrow{f} L\) in \(\Mod(A_{n+1})\) and an isomorphism \(E/xE \cong M\) of \(\pi_0(A_1)\)-modules.
    Take the fiber sequence \(\Omega_A(L) \to P \to L\) in \(\Mod(A)\) defined in (\ref{FibL}).
    Since \(P\) is projective over \(A\), there exists a map \(\map{\gamma'}{P}{E}\) such that \(P \xrightarrow{\gamma'} E \twoheadrightarrow L\) is \(P \to L\).
    Since \(M \xrightarrow{g} E \xrightarrow{f} L\) and \(\Omega_A(L) \to P \to L\) are fiber sequences, \(\gamma'\) induces a map \(\map{\gamma}{\Omega_A(L)}{M}\) of \(A\)-modules making the same diagram in (\ref{DiagramSplit}) commutative.
    As in the proof of (b) \(\Rightarrow\) (d), we obtain the same commutative diagram in (\ref{DiscreteDiagram}).
    In particular, \(f\) induces a surjective endomorphism \(\pi_0(f/^L x) \colon M \twoheadrightarrow M\) on the Noetherian \(\pi_0(A_1)\)-module \(M\) because of \(E/xE \cong M \cong L/xL\) and thus \(\pi_0(f/^L x)\) is furthermore injective.
    This shows the surjectivity of the (left lower) map \(M \cong L/xL \to M\) in (\ref{DiscreteDiagram}) and thus this is also an isomorphism. Set the inverse to be \(\phi\).
    By the definition of the map \(\alpha \colon M \cong L/xL \to L/^L x[-1] \to \Omega_A(L)/^L x\) in (\ref{DefAlpha}), the inverse map \(M \xrightarrow{\phi} L/xL\) gives the split \(s \colon \pi_0(\Omega_A(L)/^L x) \xrightarrow{\pi_0(\gamma/^L x)} M \xrightarrow{\phi} L/xL \cong M\) of \(\map{\pi_0(\alpha)}{M}{\pi_0(\Omega_A(L)/^L x)}\) by the left pentagon in (\ref{DiscreteDiagram}).
\end{proof}

\begin{lemma} \label{SplitLifting}
    Let \(A\) be an almost perfect animated \(\Lambda\)-algebra and let \(x\) be an element of \(\mfrakm\).
    Take a discrete finitely generated \(A_1\)-module \(M\) and its lifting \(L\) to \(A_n\).
    Then there exists an element \(\mathrm{ob}(M, L, n)\) in \(\Ext^2_{A_1}(M, M)\) such that its vanishing is equivalent to the existence of a discrete finitely generated \(A_{n+1}\)-module \(E\) which is a lifting of the \(A_n\)-module \(L\) to \(A_{n+1}\).
    In particular, if \(\Ext^2_{A_1}(M, M)\) vanishes, then there exists such an \(E\).
\end{lemma}

\begin{proof}
    Applying \Cref{DiscreteLiftingNil} for \(M\) and \(L\), we can show that \(L\) is discrete.
    By \Cref{LiftingSplits}, it suffices to show that the map \(\map{\alpha}{M}{\Omega_A(L)/^L x}\) in the fiber sequence \(\theta_L'\) (\ref{ThetaL'}) splits in \(\mathrm{h}\Mod(A_1)\).
    For the fiber sequence \(\theta_L'\), we can take the long exact sequence of \(\Ext\)-groups
    \begin{equation*}
        \Hom_{A_1}(\Omega_A(L)/^L x, M) \xrightarrow{- \circ \alpha} \Hom_{A_1}(M, M) \to \Ext^1_{A_1}(\Omega_{A_1}(M), M).
    \end{equation*}
    Set the element \(\mathrm{ob}'(M, L, n)\) in \(\Ext^1_{A_1}(\Omega_{A_1}(M), M)\) to be the image of \(\id_M \in \Hom_{A_1}(M, M)\).
    The existence of a splitting of \(\alpha\) is equivalent to the vanishing of \(\mathrm{ob}'(M, L, n)\) by the long exact sequence.
    By the fiber sequence (\ref{SyzygyM}), we have a fiber sequence of \(A_1\)-modules
    \begin{equation*}
        \Ext^1_{A_1}(P/^L x, M) \to \Ext^1_{A_1}(\Omega_{A_1}(M), M) \to \Ext^2_{A_1}(M, M).
    \end{equation*}
    The free \(A_1\)-module \(P/^L x = (\oplus A)/^L x \cong \oplus A_1\) is projective over \(A_1\). The first term vanishes by \cite[Proposition 7.2.2.6 and Proposition 7.2.2.7]{lurie2017Higher}.
    So we have an isomorphism \(\Ext^1_{A_1}(\Omega_{A_1}(M), M) \xrightarrow{\cong} \Ext^2_{A_1}(M, M)\).
    Set \(\mathrm{ob}(M, L, n)\) as the image of \(\mathrm{ob}'(M, L, n) \in \Ext^1_{A_1}(\Omega_{A_1}(M), M)\) in \(\Ext^2_{A_1}(M, M)\).
    By the above argument, the vanishing of \(\mathrm{ob}(M, L, n)\) is equivalent to the existence of a lifting \(E\) of \(L\) to \(A_{n+1}\).
\end{proof}

\begin{lemma} \label{LiftingTheorem}
    Let \(A\) be an almost perfect animated \(\Lambda\)-algebra and let \(x\) be an element of \(\mfrakm\).
    Let \(M\) be a discrete finitely generated \(A_1\)-module.
    If \(\Ext^2_{A_1}(M, M)\) vanishes, there exists a discrete finitely generated \(A\)-module \(L\) which is a lifting of the \(A_1\)-module \(M\) to \(A\).
\end{lemma}

\begin{proof}
    Set \(L_1 \defeq M\).
    By \Cref{SplitLifting} and our assumption, there exists a discrete finitely generated \(A_2\)-module \(L_2\) such that \(L_2\) is a lifting of the \(A_1\)-module \(L_1\) to \(A_2\).
    Again, by using \Cref{SplitLifting} for \(L_2\) and \(M\), we can take a discrete finitely generated lifting \(L_3\) of the \(A_2\)-module \(L_2\) to \(A_3\).
    By repeating this process, we can take a sequence \(\{L_n\}_{n \geq 0}\) such that each \(L_n\) is a discrete finitely generated \(A_n\)-module and \(L_{n+1}\) is a lifting of the \(A_n\)-module \(L_n\) to \(A_{n+1}\).
    By \Cref{LiftingSequence}, the \(A_1\)-module \(M\) has a lifting \(L\) to \(A\).
\end{proof}



Under these preparations, we can prove the following lifting theorem.

\begin{theorem} \label{LiftingCorollary}
    Let \(A\) be an almost perfect animated \(\Lambda\)-algebra and let \(x_1, \dots, x_t\) be a sequence of elements of \(\mfrakm\).
    Set \(\Gamma_j \defeq A/^L(x_1, \dots, x_j)\) for each \(1 \leq j \leq t\).
    Let \(M\) be a discrete finitely generated \(\Gamma_t\)-module.
    If \(\Ext^2_{\Gamma_t}(M, M)\) vanishes, then there exists a discrete finitely generated \(A\)-module \(L\) which is a lifting of the \(\Gamma_t\)-module \(M\) to \(A\).
\end{theorem}

\begin{proof}
    Since \(A\) is almost perfect, so is \(\Gamma_j\) for each \(j\).
    Set \(L_t \defeq M\).
    Our assumption is that \(\Ext^2_{\Gamma_t}(L_t, L_t)\) vanishes.
    We can apply \Cref{LiftingTheorem} when we set \(M \defeq L_t\), \(A \defeq \Gamma_{t-1}\), and \(x \defeq x_t\).
    Then there exists a discrete finitely generated \(\Gamma_{t-1}\)-module \(L_{t-1}\) and an isomorphism \(L_{t-1} \otimes^L_{\Gamma_{t-1}} \Gamma_t \cong M\) of \(\Gamma_t\)-modules.

    Next, we iterate this process.
    The derived quotient \(\Gamma_t = \Gamma_{t-1}/^L x_t\) gives a fiber sequence \(L_{t-1} \xrightarrow{\times x_t} L_{t-1} \to M\) of \(\Gamma_{t-1}\)-modules.
    Applying \(\Ext^i_{\Gamma_{t-1}}(L_{t-1}, -)\) for the above fiber sequence, we have an exact sequence of \(\pi_0(\Gamma_{t-1})\)-modules
    \begin{equation} \label{ExtLongExactSequenceIterate}
        \Ext^i_{\Gamma_{t-1}}(L_{t-1}, L_{t-1}) \xrightarrow{\times x_{t-1}} \Ext^i_{\Gamma_{t-1}}(L_{t-1}, L_{t-1}) \to \Ext^i_{\Gamma_{t-1}}(L_{t-1}, L_t).
    \end{equation}
    The tensor-forgetful adjunction and our assumption give the isomorphism
    \begin{equation*}
        \Ext^2_{\Gamma_{t-1}}(L_{t-1}, L_t) \cong \Ext^2_{\Gamma_t}(L_{t-1} \otimes^L_{\Gamma_{t-1}} \Gamma_t, L_t) \cong \Ext^2_{\Gamma_t}(L_t, L_t) = \Ext^2_{\Gamma_t}(M, M) = 0.
    \end{equation*}
    Since \(\pi_0(\Gamma_{t-1}) \cong \pi_0(A)/(x_1, \dots, x_{t-1})\) is an \(\mfrakm\)-adically complete Noetherian local ring and \(L_{t-1}\) is a discrete finitely generated \(\Gamma_{t-1}\)-module, the base change \(L_{t-1} \otimes^L_{\Gamma_{t-1}} \pi_0(\Gamma_{t-1})\) is almost perfect over \(\pi_0(\Gamma_{t-1})\) by \Cref{LemBaseChangePerfect} (b). So the \(\Ext\)-group \(\Ext^2_{\pi_0(\Gamma_{t-1})}(L_{t-1} \otimes^L_{\Gamma_{t-1}} \pi_0(\Gamma_{t-1}), L_{t-1}) \cong \Ext^2_{\Gamma_{t-1}}(L_{t-1}, L_{t-1})\) is also a discrete finitely generated \(\pi_0(\Gamma_{t-1})\)-module (this also follows from \Cref{AlmostPerfMapFin}).
    By Nakayama's lemma, the \(\Ext\)-group \(\Ext^2_{\Gamma_{t-1}}(L_{t-1}, L_{t-1})\) vanishes.

    Inductively, we can take a discrete and finite \(\Gamma_j\)-module \(L_j\) and an isomorphism \(L_j \otimes^L_{\Gamma_j} \Gamma_{j+1} \cong L_{j+1}\) of \(\Gamma_{j+1}\)-modules for each \(0 \leq j \leq t-1\) (we set \(\Gamma_0 \defeq A\)).
    The isomorphism
    \begin{align*}
        L_0 \otimes^L_{\Gamma_0} \Gamma_t & \cong L_0 \otimes^L_{\Gamma_0} \Gamma_1 \otimes^L_{\Gamma_1} \Gamma_t \cong L_1 \otimes^L_{\Gamma_1} \Gamma_t \\
        & \cong \cdots \cong L_{t-1} \otimes^L_{\Gamma_{t-1}} \Gamma_t \cong L_t = M
    \end{align*}
    of \(\Gamma_t\)-modules shows that the \(\Gamma_0\)-module \(L \defeq L_0\) is a lifting of the \(\Gamma_t\)-module \(M\) to \(A\).
\end{proof}

\section{Rerevance to Auslander's Zero-Divisor Theorem} \label{SectionZeroDivisor}

Once we try to apply \Cref{LiftingCorollary} for a regular local ring \(A\), the result is covered by Auslander--Ding--Solberg's result \cite{auslander1993Liftings}.
In this section, we show this result (\Cref{EquivLCI}) by combining our result (\Cref{ColimitPreserving}) and Auslander's zero-divisor theorem (\Cref{AuslanderZeroDivisor}).

Auslander \cite{auslander1961Modules} proposed the following so-called Auslander's zero-divisor conjecture, which already has been proved.

\begin{theorem}[{\cite{auslander1961Modules}}] \label{AuslanderZeroDivisor}
    Let \(A\) be a Noetherian local ring and let \(M\) be a non-zero discrete finitely generated \(A\)-module with finite projective dimension.
    If a sequence of elements \(f_1, \dots, f_r\) of \(A\) is a regular sequence on \(M\), then \(f_1, \dots, f_r\) is also a regular sequence on \(A\).
\end{theorem}

If \(A\) is a regular local ring, this is a consequence of Serre's intersection theorem \cite{serre1975Algebre}.
The general statement follows from Peskine--Szpiro's intersection theorem which is developed and solved by Peskine--Szpiro \cite{peskine1969Topologie,peskine1973Dimension} in positive characteristic or essentially finite type over a field, Hochster \cite{hochster1974Equicharacteristic} in equicharacteristic, and Roberts \cite{roberts1987Theoreme} in mixed characteristic.
See also a survey \cite{roberts1989Intersection} or a book \cite[Theorem 6.2.3]{roberts1998Multiplicities}.
Moreover, this theorem is one of the members of the initial homological conjecture and thus is now a consequence of, in particular, direct summand conjecture (see, for example, \cite[\S 1.6]{dospinescu2023Conjecture}).

By using this, we can show the ``full circle'' result.

\begin{corollary} \label{EquivLCI}
    Let \(A\) be a regular local ring and let \(f_1, \dots, f_r\) be a sequence of elements of \(A\).
    Set \(A/^L\underline{f} = A/^L(f_1, \dots, f_r)\) and \(R \defeq A/(f_1, \dots, f_r)\)
    If there exists a discrete finitely generated non-zero \(R\)-module \(M\) such that \(\Ext^2_{A/^L\underline{f}}(M, M) = 0\), then \(f_1, \dots, f_r\) is a regular sequence of \(A\). 
\end{corollary}

\begin{proof}
    If \(f_1, \dots, f_r\) is a regular sequence of \(A\), then \(R\) itself can be taken as a desired \(R\)-module \(M\).
    Conversely, by \Cref{LiftingCorollary}, there exists a discrete finitely generated non-zero \(A\)-module \(L\) such that \(L \otimes^L_A A/^L\underline{f} \cong M\). Since \(A\) is regular local, \(L\) has a finite projective dimension over \(A\).
    The derived tensor product \(M \cong L \otimes^L_A A/^L\underline{f} \cong L/^L (f_1, \dots, f_r)\) is discrete and thus \(f_1, \dots, f_r\) is a (Koszul-)regular sequence on \(L\) (see, for example, \citeSta{09CC}).
    The Auslander's zero-divisor theorem (\Cref{AuslanderZeroDivisor}) shows that \(f_1, \dots, f_r\) is a regular sequence on \(A\).
    We are done.
\end{proof}

This result was not mentioned in \cite{nasseh2013Liftings}, in which they proved a similar result of \Cref{LiftingCorollary} by using DG methods.

\begin{remark} \label{AssumptionDerived}
    The original assumption of the Auslander--Reiten conjecture (\Cref{AuslanderReiten}) is the vanishing of \(\Ext^i_{R}(M, M \oplus R)\) for each \(i \geq 1\).
    On the contrary, under the setting of \Cref{MainTheorem2Derived} (or more generally \Cref{LiftingCorollary}),  our assumption should be \(\Ext^i_{A/^L\underline{f}}(M, M \oplus R) = 0\) for each \(i \geq 1\).
    However, because of \Cref{EquivLCI} (this is nontrivial if we do not have our main theorem (\Cref{LiftingCorollary})), our assumption implies that \(f_1, \dots, f_r\) is a regular sequence on \(A\) and thus those two assumptions are (trivially) equivalent.

    Moreover, by \Cref{AssumptionLemma2} and \Cref{EquivLCI}, for a non-regular sequence \(f_1, \dots, f_r\) on the regular local ring \(A\), the vanishing \(\Ext^i_R(M, M)\) is strictly weaker than the vanishing \(\Ext^i_{A/^L\underline{f}}(M, M)\) for each \(i \geq 1\) since the latter vanishing implies the regularity of \(f_1, \dots, f_r\) on \(A\).
\end{remark}

\end{document}